\documentclass[12pt,twoside]{amsart}
\usepackage{amssymb}
\usepackage{amscd}

\title[The trace map of Frobenius]
{The trace map of Frobenius and extending sections for threefolds} 
\author{Hiromu Tanaka} 
\subjclass[2010]{14E30, 13A35.}
\keywords{minimal model program, basepoint freeness, extension theorem, positive characteristic}
\address{Department of Mathematics, Imperial College, London, 180 Queen's Gate, 
London SW7 2AZ, UK} 
\email{h.tanaka@imperial.ac.uk}
\newcommand{\Ker}[0]{{\operatorname{Ker}}}
\newcommand{\Image}[0]{{\operatorname{Image}}}
\newcommand{\Tr}[0]{{\operatorname{Tr}}}

\newcommand{\Spec}[0]{{\operatorname{Spec}}}

\newcommand{\Supp}[0]{{\operatorname{Supp}}}

\newcommand{\Ex}[0]{{\operatorname{Ex}}}
\newtheorem{thm}{Theorem}[section]
\newtheorem{lem}[thm]{Lemma}
\newtheorem{cor}[thm]{Corollary}
\newtheorem{prop}[thm]{Proposition}

\newtheorem{fact}[thm]{Fact}
\newtheorem{ex}[thm]{Example}

\theoremstyle{definition}
\newtheorem{ques}[thm]{Question}
\newtheorem{conje}[thm]{Conjecture}
\newtheorem{dfn}[thm]{Definition}

\newtheorem{rem}[thm]{Remark}
\newtheorem*{ack}{Acknowledgments}      
         
\newtheorem{step}{Step}
\newtheorem{nasi}[thm]{}

\newcommand{\Q}{$\mathbb{Q}$}
\newcommand{\Z}{$\mathbb{Z}$}
\newcommand{\MO}{\mathcal{O}}

\begin{document}

\maketitle

\begin{abstract}
In this paper, by using the trace map of Frobenius, 
we consider problems on extending sections for positive characteristic threefolds. 
\end{abstract}

\tableofcontents

\setcounter{section}{-1}

\section{Introduction}

In characteristic zero, by the Kodaira vanishing theorem and its generalizations, 
we can establish some results on adjoint divisors, such as 
the Kawamata--Shokurov basepoint free theorem (see, for example, \cite[Theorem~3.3]{KM}) and 
the Hacon--M\textsuperscript{c}Kernan extension theorem (\cite[Theorem~5.4.21]{HM}). 
These theorems claim, under suitable conditions,  that
an adjoint divisor 
$m(K_X+\Delta+A)$ has good properties, where 
$m\in\mathbb Z_{>0}$, 
$(X, \Delta)$ is a pair and $A$ is an ample divisor. 
In this paper we only consider the following very simple situation: 
$X$ is a smooth projective variety, 
$\Delta=S$ is a smooth prime divisor and $A$ is an ample Cartier divisor. 
The following fact immediately follows from the Kodaira vanishing theorem. 

\begin{fact}\label{0char-zero}
Let $k$ be an algebraically closed field of characteristic zero. 
Let $X$ be a smooth projective variety over $k$. 
Let $S$ be a smooth prime divisor on $X$ and 
let $A$ be an ample Cartier divisor on $X$ such that $K_X+S+A$ is nef. 
Fix $m\in\mathbb Z_{>0}.$ 
Then, by the Kodaira vanishing theorem, we obtain 
$$H^1(X, K_X+A+(m-1)(K_X+S+A))=0.$$
Thus, the natural restriction map 
$$H^0(X, m(K_X+S+A))\to H^0(S, m(K_S+A|_S))$$
is surjective. 
\end{fact}
 
It is natural to consider whether 
the above fact also holds in positive characteristic. 
Unfortunately, however, there exists the following example. 

\begin{ex}[cf. Example~\ref{cex-surjective}]
Let $k$ be an algebraically closed field of positive characteristic. 
Then, there exist a smooth projective surface $X$ over $k$, 
a smooth prime divisor $C$ on $X$, and 
an ample Cartier divisor $A$ on $X$ 
such that $K_X+C+A$ is nef and the natural restriction map 
$$H^0(X, K_X+C+A)\to H^0(C, K_C+A|_C)$$
is not surjective. 
\end{ex}

Thus, we would like to find a suitable analogue of Fact~\ref{0char-zero} in positive characteristic. 
In this paper, we prove the following two theorems. 

\begin{thm}[cf. Corollary~\ref{surface-restriction}]\label{0surface-special-extension}
Let $k$ be an algebraically closed field of positive characteristic. 
Let $X$ be a smooth projective surface over $k$. 
Let $C$ be a smooth prime divisor on $X$ and 
let $A$ be an ample Cartier divisor on $X$. 
If $H^0(C, K_C+A|_C)\neq 0$, then 
the natural restriction map 
$$H^0(X, K_X+C+A)\to H^0(C, K_C+A|_C)$$
is a nonzero map. 
\end{thm}

\begin{thm}[cf. Theorem~\ref{threefold-extension}]\label{0threefold-extension}
Let $k$ be an algebraically closed field of positive characteristic. 
Let $X$ be a smooth projective threefold over $k$. 
Let $S$ be a smooth prime divisor on $X$ and let $A$ be an ample Cartier divisor on $X$. 
Assume the following two conditions. 
\begin{enumerate}
\item{$K_X+S+A$ is nef.}
\item{$\kappa(S, K_S+A|_S)\neq 0$.}
\end{enumerate}
Then, there exists $m_0\in \mathbb Z_{>0}$ such that, for every integer $m\geq m_0$, 
the natural restriction map 
$$H^0(X, m(K_X+S+A))\to H^0(S, m(K_S+A|_S))$$
is surjective.
\end{thm}

To show the above two theorems, we use the trace map of Frobenius. 
This strategy is essentially the same as that used in \cite[Proposition~5.3]{Schwede2} and its proof. 
Let us see the idea of the proofs. 
Let $X$ be a smooth projective variety. 
Let $S$ be a smooth prime divisor on $X$ and 
let $A$ be an ample Cartier divisor on $X$. 
Then, for every $e\in\mathbb Z_{>0}$, 
we obtain the following commutative diagram by using the trace map of Frobenius: 
{\small 
$$\begin{CD}
H^0(X, K_X+S+p^eA)@>>>H^0(S, K_S+p^eA|_S)@>>>H^1(X, K_X+p^eA)\\
@VVV @VV\Tr_S^e(A|_S)V\\
H^0(X, K_X+S+A)@>\rho >> H^0(S, K_S+A|_S)
\end{CD}$$}
where the lower horizontal arrow $\rho$ is the natural restriction map and 
the upper horizontal sequence is exact. 
By the Serre vanishing theorem, for large $e\gg 0$, 
we obtain the vanishing $H^1(X, K_X+p^eA)=0$. 
Thus, to prove that the restriction map $\rho$ is surjective (resp. a nonzero map), 
it is sufficient to show that the trace map $\Tr_S^e(A|_S)$ 
is surjective (resp. a nonzero map). 
Therefore, to prove the above two theorems, 
we establish the following results on the trace map of Frobenius. 

\begin{thm}[cf. Theorem~\ref{curve-trace-nonzero}]\label{0trace-curve}
Let $k$ be an algebraically closed field of positive characteristic. 
Let $C$ be a smooth projective curve over $k$. 
Let $A$ be an ample Cartier divisor on $C$. 
If $H^0(C, K_C+A)\neq 0$, then the trace map 
$$\Tr_C^e(A):H^0(C, K_C+p^eA)\to H^0(C, K_C+A)$$
is a nonzero map for every $e\in\mathbb Z_{>0}$. 
\end{thm}

\begin{thm}[cf. Theorem~\ref{trace-surface}]\label{0trace-surface}
Let $k$ be an algebraically closed field of positive characteristic. 
Let $S$ be a smooth projective surface over $k$. 
Let $A$ be an ample Cartier divisor on $S$. 
Assume the following two conditions. 
\begin{enumerate}
\item{$K_S+A$ is nef.}
\item{$\kappa(S, K_S+A)\neq 0$.}
\end{enumerate}
Then, there exists $m_1\in\mathbb Z_{>0}$ such that 
the trace map 
$$\Tr_S^e(A+m(K_S+A)):$$
$$H^0(S, K_S+p^e(A+m(K_S+A)))\to H^0(S, K_S+(A+m(K_S+A)))$$
is surjective for every integer $m\geq m_1$ and for every $e\in\mathbb Z_{>0}$. 
\end{thm}

We also consider whether 
Theorem~\ref{0threefold-extension} and Theorem~\ref{0trace-surface} hold for 
the case where $\kappa(S, K_S+A)=0$. 
Let us compare Theorem~\ref{0threefold-extension} with
the following basepoint free conjecture (cf. \cite[Theorem~3.3]{KM}).

\begin{conje}\label{0bpf}
Let $k$ be an algebraically closed field of positive characteristic. 
Let $X$ be a smooth projective threefold over $k$. 
Let $S$ be a smooth prime divisor on $X$ and 
let $A$ be an ample Cartier divisor on $X$ such that 
$K_X+S+A$ is nef. 
Then, $|b(K_X+S+A)|$ is basepoint free for every $b\gg 0$. 
\end{conje}

\begin{rem}
After a first version of this paper was circulated, 
\cite[Theorem~1.2]{BW} proved that, in the above situation, 
$|b(K_X+S+A)|$ is basepoint free for every sufficiently large and divisible $b$ in characteristic $p>5$. 
\end{rem}

If Conjecture~\ref{0bpf} holds, then 
Theorem~\ref{0threefold-extension} also holds when $\kappa(S, K_S+A)=0$. 
It is natural to ask whether Theorem~\ref{0trace-surface} also holds when $\kappa(S, K_S+A)=0$. 
Unfortunately, the answer is negative. 
We can construct the following example in characteristic two.  

\begin{thm}[cf. Theorem~\ref{trace-cex}]\label{0trace-surface-cex}
Let $k$ be an algebraically closed field of characteristic two. 
Then, there exists a smooth projective surface $S$ over $k$ 
such that 
\begin{enumerate}
\item{$-K_S=:A$ is ample. In particular $\kappa(S, K_S+A)=0$. }
\item{For  every $e\in\mathbb Z_{>0}$, 
the trace map 
$$\Tr_S^e(A):H^0(S, K_S+2^eA)\to H^0(S, K_S+A)$$
is the zero map. }
\end{enumerate}
\end{thm}

Moreover, Theorem~\ref{0trace-surface-cex} also shows that 
we can not generalize Theorem~\ref{0trace-curve} to dimension two.

In the appendix to this paper (Section~9), we establish the following analogue of 
the Hacon--M\textsuperscript{c}Kernan extension theorem for surfaces. 

\begin{thm}[cf. Theorem~\ref{surface-ext}]\label{0surface-ext}
Let $k$ be an algebraically closed field of positive characteristic. 
Let $X$ be a smooth projective surface over $k$ and 
let $C$ be a smooth prime divisor on $X$. 
Let $\Delta:=C+B$, where 
$B$ is an effective \Q-divisor on $X$ which satisfies the following properties: 
\begin{enumerate}
\item{$C\not\subset\Supp B$, $\llcorner B\lrcorner=0$ and 
$(X, \Delta)$ is plt, }
\item{$B\sim_{\mathbb Q}A+F$, where $A$ is an ample \Q-divisor and 
$F$ is an effective \Q-divisor such that $C\not\subset \Supp F$, and}
\item{No prime component of $\Delta$ is contained in the stable base locus of $K_X+\Delta.$}
\end{enumerate}
Then, there exists an integer $m_0>0$ such that, for every integer $m>0$, 
the restriction map 
$$H^0(X, mm_0(K_X+\Delta))\to H^0(C, mm_0(K_X+\Delta)|_C)$$
is surjective.
\end{thm}

However, the proof of Theorem~\ref{0surface-ext} 
does not use the trace map of Frobenius. 
We use instead results on the minimal model theory established in \cite{T2} and \cite{T3}.

\begin{nasi}[Overview of contents]
In Section~1, we summarize the notation. 
In Section~2, we give the definition and some basic properties of the trace map of Frobenius. 
The trace map of Frobenius is obtained by applying the functor $\mathcal Hom_{\mathcal O_X}(-, \omega_X)$ 
to the Frobenius map $\mathcal O_X\to F_*\mathcal O_X.$ 
In Section~3, we recall some known facts about the Cartier operator. 
We can consider the trace map of Frobenius as a Cartier operator. 
The Cartier operator is defined by the de Rham complex. 
We use the Cartier operator to consider the relation between 
the trace map of Frobenius and \'etale base change. 
In Section~4, we prove Theorem~\ref{0surface-special-extension} and Theorem~\ref{0trace-curve}. 
In Section~5, we prove Theorem~\ref{0trace-surface} when $\kappa=1$. 
In Section~6, we prove Theorem~\ref{0trace-surface} when $\kappa=2$. 
In Section~7, by using Theorem~\ref{0trace-surface}, we show Theorem~\ref{0threefold-extension}. 
In Section~8, we prove Theorem~\ref{0trace-surface-cex}. 
In Section~9, we prove Theorem~\ref{0surface-ext}. 
\end{nasi}

\begin{nasi}[Overview of related literature] 
We give some general references to the literature related to this paper in connection with 
the basepoint free theorem, the extension theorem, the trace map of Frobenius 
and the minimal model theory in positive characteristic. 

{\em Basepoint free theorem and extension theorem}. 
The motivation of this paper comes from the basepoint free theorem and the extension theorem 
in characteristic zero. 
Thus, let us summarize some known results on this topic. 
Kawamata and Shokurov established the basepoint free theorem for klt pairs 
(cf. \cite{KMM}, \cite{KM}). 
\cite{Ambro} generalized this result (cf. \cite{Fujino}). 
The extension theorem is established by Hacon and M\textsuperscript{c}Kernan (\cite[Theorem~5.4.21]{HM}). 
This theorem is a key to proving the existence of flips (\cite{BCHM}). 
For related topics, see \cite{DHP} and \cite{FG}. 

{\em The trace map of Frobenius}.  
At the heart of this paper is the trace map of Frobenius. 
This map plays a crucial role in the theory of $F$-singularities 
(cf. \cite{BST}, \cite{Schwede1}, \cite{Schwede2}). 
Moreover, \cite{CHMS} and \cite{Mustata} establish results related to birational geometry 
by using the trace map of Frobenius and the theory of $F$-singularities. 
For related topics, see \cite{BSTZ} and \cite{TW}. 
 
{\em Minimal model theory in positive characteristic}. 
For results on the minimal model theory in positive characteristic, we refer to 
\cite{Fujita3}, \cite{KK}, \cite{T2} and \cite{T3} for the case of surfaces and to 
\cite{Birkar}, \cite{BW}, \cite{CTX}, \cite{HX}, \cite{Kawamata}, \cite{Keel}, \cite{Kollar} and \cite{Xu} 
for the case of threefolds.  
\end{nasi}

\begin{ack}
The author would like to 
thank Professor Osamu Fujino for many comments and discussions. 
He would also like to thank Professor Karl Schwede for numerous helpful comments. 
He thanks Professor Atsushi Moriwaki for warm encouragement. 
He also thanks the referee for many valuable comments and suggestions. 
The author was partially supported by 
the Research Fellowships of the Japan Society for the Promotion of Science for Young Scientists (24-1937). 
\end{ack}

\section{Notation}
We freely use the notation and terminology in \cite{KM}. 
We do not distinguish in notation between 
invertible sheaves and divisors. 
For example, 
we often write $L+M$ for 
invertible sheaves $L$ and $M$. 
For a coherent sheaf $F$ and a Cartier divisor $L$, 
we define $F(L):=F\otimes \mathcal O_X(L)$. 

Throughout this paper, 
we work over an algebraically closed field $k$ 
of positive characteristic and 
let ${\rm char}\,k=:p>0.$
In this paper, {\em a variety} means 
an integral scheme, separated and of finite type over $k$. 
A {\em curve} (resp. a {\em surface}) means a variety of dimension one (resp. two).

\section{The trace map of Frobenius}

In this section, we define the trace map of Frobenius and 
we discuss some fundamental properties. 
We only use the smooth case. 
For the singular case, see \cite[Section~2]{Schwede2}. 

\begin{prop}\label{tracedef}
Let $X$ be a smooth projective variety. 
Let $E$ be an effective \Z-divisor and 
let $D$ be a \Z-divisor. 
Then, for every positive integer $e$, 
there exists a natural $\mathcal O_X$-module homomorphism 
$$\Tr_{X,E}^e(D):F^e_*(\omega_X(E+p^eD))\to 
\omega_X(E+D).$$
We call this a trace map. 
\end{prop}

\begin{proof}
Consider the Frobenius map 
$$\mathcal O_X\to F^e_*\mathcal O_X,$$ 
that is, the $p^e$-th power map $a\mapsto a^{p^e}$. 
Since $E$ is effective, we obtain by tensoring with $\MO_X(-E)$
$$\mathcal O_X(-E)\to F^e_*(\mathcal O_X(-p^eE))\hookrightarrow F^e_*(\mathcal O_X(-E)).$$ 
Tensoring with $\mathcal O_X(-D)$, we obtain 
\begin{eqnarray*}
\mathcal O_X(-E-D)
&\to& F^e_*(\mathcal O_X(-E))\otimes \mathcal O_X(-D)\\
&\simeq& F^e_*(\mathcal O_X(-E-p^eD)).
\end{eqnarray*}
By the duality theorem for finite morphisms, 
we obtain 
$$\mathcal Hom_{\mathcal O_X}(F^e_*(\mathcal O_X(-E-p^eD)), \omega_X)\simeq F^e_*(\omega_X(E+p^eD)).$$ 
Then, we apply the functor 
$\mathcal Hom_{\mathcal O_X}(-, \omega_X)$ and we obtain 
$$F^e_*(\omega_X(E+p^eD))\to \omega_X(E+D).$$ 
This is the trace map $\Tr_{X,E}^e(D)$.  
\end{proof}

\begin{rem}\label{pE-vs-E}
By the above construction, 
$\Tr_{X,E}^e(D)$ factors through $\Tr_X^e(E+D):=\Tr_{X, 0}^e(E+D)$: 
{\small $$\Tr_{X,E}^e(D):F^e_*(\omega_X(E+p^eD))\hookrightarrow F^e_*(\omega_X(p^eE+p^eD))
\xrightarrow{\Tr_X^e(E+D)} \omega_X(E+D).$$}
\end{rem}

\begin{rem}\label{trace-local}
Let $X$ be a smooth projective variety. 
Let $\Spec R\subset X$ be an affine open subset such that 
$R$ has a $p$-basis $\{x_1, \cdots, x_n\}.$ 
Then, we obtain 
$$\Gamma(\Spec R, \omega_X)=Rdx_1\wedge\cdots\wedge dx_n$$
and 
\begin{eqnarray*}
\Gamma(\Spec R, F_*^e\omega_X)
&=&\bigoplus_{0\leq i_j<p^e}R^{p^e}x_1^{i_1}\cdots x_n^{i_n}dx_1\wedge\cdots\wedge dx_n.
\end{eqnarray*}
The trace map 
$$\Tr_X^e:\Gamma(\Spec R, F_*^e\omega_X)\to \Gamma(\Spec R, \omega_X)$$
is described as follows: 
\begin{enumerate}
\item{$\Tr_X^e(x_1^{p^e-1}\cdots x_n^{p^e-1}dx_1\wedge\cdots\wedge dx_n)=dx_1\wedge\cdots\wedge dx_n.$}
\item{$\Tr_X^e(x_1^{i_1}\cdots x_n^{i_n}dx_1\wedge\cdots\wedge dx_n)=0$ if $0\leq i_j<p^e-1$ for some 
$1\leq j\leq n$.}
\end{enumerate}
\end{rem}

The following two lemmas give some fundamental properties. 

\begin{lem}\label{D-linearequ}
Let $X$ be a smooth projective variety and 
let $E$ be an effective \Z-divisor. 
If $D_1$ and $D_2$ are linearly equivalent \Z-divisors, 
then the two trace maps 
$\Tr_X^e(D_1)$ and $\Tr_X^e(D_2)$ are the same 
for every positive integer $e$, that is, 
there exists a commutative diagram: 
$$\begin{CD}
F^e_*(\omega_X(E+p^eD_2))@>\Tr_X^e(D_2)>> \omega_X(E+D_2)\\
@VV\simeq V @VV\simeq V\\
F^e_*(\omega_X(E+p^eD_1))@>\Tr_X^e(D_1)>> \omega_X(E+D_1).
\end{CD}$$ 
\end{lem}

\begin{proof}
The assertion follows from $\Tr_X^e(D_i) =\Tr_X^e\otimes_{\MO_X} \MO_X(D_i)$. 
\end{proof}

\begin{lem}\label{e-vs-e+1}
Let $X$ be a smooth projective variety and 
let $E$ be an effective \Z-divisor. 
Let $D$ be a \Z-divisor. 
Then, for every positive integer $e$, 
$$\Tr_{X, E}^{e+1}(D)=\Tr_{X, E}^e(D)\circ F_*^e(\Tr_{X, E}^1(p^e D)),$$
that is,  
\begin{eqnarray*}
\Tr_{X, E}^{e+1}(D):
F^{e+1}_*(\omega_X(E+p^{e+1}D))&\xrightarrow{F_*^e(\Tr_{X,E}^1(p^eD))}&F^e_*(\omega_X(E+p^eD))\\
&\xrightarrow{\Tr_{X,E}^{e}(D)}&\omega_X(E+D).
\end{eqnarray*}
\end{lem}

\begin{proof}
Consider the Frobenius maps 
$$\mathcal O_X(-E)\to F^e_*(\mathcal O_X(-E))\to F^{e+1}_*(\mathcal O_X(-E)).$$ 
Tensoring with $\mathcal O_X(-D)$, we obtain 
\begin{eqnarray*}
\mathcal O_X(-E-D)
\to F^e_*(\mathcal O_X(-E-p^eD))\to 
F^{e+1}_*(\mathcal O_X(-E-p^{e+1}D)).
\end{eqnarray*}
Applying the functor 
$\mathcal Hom_X(-, \omega_X)$, we obtain the assertion. 
\end{proof}

In this paper, 
we often use the following two commutative diagrams. 

\begin{lem}\label{trace-diagram}
Let $X$ be a smooth projective variety and 
let $S$ be a smooth prime divisor. 
Then, there exist the following commutative diagrams: 
\begin{enumerate} 
\item{$$\begin{CD}
0@>>> F_*^e\omega_X@>>> F_*^e(\omega_X(S))@>>> F_*^e\omega_S@>>> 0\\
@. @VV\Tr_X^e V @VV\Tr_{X,S}^e V @VV\Tr_S^e V\\
0@>>> \omega_X@>>> \omega_X(S)@>>> \omega_S@>>> 0, 
\end{CD}$$}
\item{$$\begin{CD}
0@>>> F_*^e\omega_X@>>> F_*^e(\omega_X(p^eS))@>>> F_*^e(\omega_{p^eS})@>>> 0\\
@. @VV\Tr_X^e V @VV\Tr_{X}^e(S) V @VV\Psi V\\
0@>>> \omega_X@>>> \omega_X(S)@>>> \omega_S@>>> 0.
\end{CD}$$
Moreover, $\Tr_S^e$ factors through $\Psi$$:$ 
$$\Tr_S^e:F_*^e\omega_S\to F_*^e(\omega_{p^eS})\overset{\Psi}\to \omega_S.$$
}
\end{enumerate} 
\end{lem}

\begin{proof}
(1) Consider the following commutative diagram: 
$$\begin{CD}
0@>>> F_*^e(\mathcal O_X(-S))@>>> F_*^e\mathcal O_X@>>> F_*^e\mathcal O_S@>>> 0\\
@. @AAA @AA F_X^e A @AAF_S^e A\\
0@>>> \mathcal O_X(-S)@>>> \mathcal O_X@>>> \mathcal O_S@>>> 0.
\end{CD}$$
Apply the functor 
$\mathcal Hom_X(-, \omega_X)$ and we obtain the assertion. \\
(2) 
 Consider the following commutative diagram: 
$$\begin{CD}
0@>>> F_*^e(\mathcal O_X(-p^eS))@>>> F_*^e\mathcal O_X@>>> F_*^e\mathcal O_{p^eS}@>>> 0\\
@. @AAA @AA F_X^e A @AAA\\
0@>>> \mathcal O_X(-S)@>>> \mathcal O_X@>>> \mathcal O_S@>>> 0.
\end{CD}$$
Apply the functor 
$\mathcal Hom_X(-, \omega_X)$ and we obtain the required commutative diagram. 
Since 
$$F_S^e:\mathcal O_S\to F_*^e\mathcal O_S$$
factors through $F_*^e\mathcal O_{p^eS}$, 
we obtain the last assertion in the lemma. 
\end{proof}

\begin{rem}\label{Tango-cex}
Given a smooth projective variety $X$ and an ample \Z-divisor $A$ on $X$, 
it is natural to ask whether for every $e\in\mathbb Z_{>0}$, 
the trace map $\Tr_X^e(A)$ is surjective. 
However, this has a negative answer. 
Indeed, \cite{Tango} constructs a smooth projective curve $X$ and 
an ample \Z-divisor $A$ on $X$ such that the trace map $\Tr_X^e(A)$ is not surjective for $e=1$. 
\end{rem}

On the other hand, 
we obtain an affirmative answer for the following two cases: 
abelian varieties and $F$-split varieties.

\begin{prop}
If $X$ is an abelian variety and $A$ is an ample \Z-divisor, then the trace map 
$$\Tr_{X}^e(A):H^0(X, \omega_X(p^eA))\to H^0(X, \omega_X(A)).$$
is surjective for every $e\in\mathbb Z_{>0}$. 
\end{prop}

\begin{proof}
Fix $e\in\mathbb Z_{>0}$. 
For $m\in\mathbb Z$, let $m_X:X \to X$ be the $m$-multiplication map of the abelian variety $X$. 
If $n\in\mathbb Z_{>0}$ is not divisible by $p$, then 
$$n_X:X\to X$$
is a finite morphism whose degree is not divisible by $p$. 
Thus, 
$$\mathcal O_X\to (n_X)_*\mathcal O_X$$ 
is split as an $\mathcal O_X$-module homomorphism (cf. \cite[Proposition~5.7(2)]{KM}). 
We obtain the following commutative diagram 
$$\begin{CD}
F_*^e(n_X)_*(\mathcal O_X(p^e(n_X)^*A))@<<< (n_X)_*(\mathcal O_X((n_X)^*A))\\
@AAA @AA \widetilde{n_X} A\\
F_*^e(\mathcal O_X(p^eA))@<<<\mathcal O_X(A)
\end{CD}$$
Applying the functor $\mathcal Hom_{\mathcal O_X}(-, \omega_X)$ 
(cf. the proof of Proposition~\ref{tracedef}) and taking global sections gives: 
$$\begin{CD}
H^0(X, \omega_X(p^e(n_X)^*A))@>\Tr_{X}^e((n_X)^*A)>> H^0(X, \omega_X((n_X)^*A))\\
@VVV @VV \widetilde{n_X}' V\\
H^0(X, \omega_X(p^eA))@>\Tr_{X}^e(A)>> H^0(X, \omega_X(A)).
\end{CD}$$
Here, $\widetilde{n_X}'$ is surjective by the splitting of $\mathcal O_X\to (n_X)_*\mathcal O_X$. 
Therefore, it is sufficient to show that $\Tr_{X}^e((n_X)^*A)$ is surjective. 
By \cite[Corollary~3 in Section~6]{Mumford}, we obtain 
$$n_X^*A= \frac{n^2+n}{2}A+\frac{n^2-n}{2}(-1)_X^*A.$$
Note that, since $(-1)_X$ is an automorphism, 
$(-1)_X^*A$ is ample. 
Therefore, by the Fujita vanishing theorem (\cite[Theorem~(1)]{Fujita1}, 
\cite[Section~5]{Fujita2}), 
we can find $n\in\mathbb Z_{>0}$ such that 
$\Tr_{X}^e((n_X)^*A)$ is surjective. 
\end{proof}

\begin{dfn}
Let $X$ be a smooth projective variety. 
We say $X$ is {\em F-split} if 
the Frobenius map 
$$\mathcal O_X\to F_*\mathcal O_X$$
is split as an $\mathcal O_X$-module homomorphism. 
\end{dfn}

\begin{prop}\label{F-surje}
Let $X$ be an $F$-split smooth projective variety and 
let $D$ be a \Z-divisor. 
Then, the trace map 
$$\Tr_{X}^e(D):H^0(X, \omega_X(p^eD))\to H^0(X, \omega_X(D)).$$
is surjective for every $e\in\mathbb Z_{>0}$. 
\end{prop}

\begin{proof}
By the definition of $F$-splitting, we see that 
the Frobenius map 
$$\mathcal O_X(-D)\to F_*^e(\mathcal O_X(-p^eD))$$
is split. 
Applying the functor $\mathcal Hom_{\mathcal O_X}(-, \omega_X)$, we obtain the assertion. 
\end{proof}

\section{Facts on Cartier operator}

In this section, we collect some facts on the Cartier operator. 
By Remark~\ref{cartier-trace}, 
we consider the trace map of Frobenius as the Cartier operator.

\begin{dfn}\label{de-rham}
Let $X$ be a smooth variety. 
Consider the de Rham complex of $X$  
$$\Omega_X^{\bullet}:\mathcal O_X\overset{d_0}\to\Omega^1_X\overset{d_1}\to\Omega^2_X
\overset{d_2}\to\cdots$$
where $\Omega^i_X:=\Omega^i_{X/k}.$
Apply $F_*$ and we obtain a complex 
$$F_*\Omega_X^{\bullet}:F_*\mathcal O_X\overset{F_*d_0}\to F_*\Omega^1_X\overset{F_*d_1}\to
F_*\Omega^2_X\overset{F_*d_2}\to\cdots.$$
Then, it is easy to see that 
$F_*d_i$ is an $\mathcal O_X$-module homomorphism. 
We define 
\begin{eqnarray*}
B_X^i&:=&\Image(F_*d_{i-1}:F_*\Omega_X^{i-1}\to F_*\Omega_X^{i})\\
Z_X^i&:=&\Ker(F_*d_i:F_*\Omega_X^i\to F_*\Omega_X^{i+1}).
\end{eqnarray*}
Note that $B_X^i$ and $Z_X^i$ are coherent sheaves. 
\end{dfn}

\begin{fact}\label{cartier-def}
Let $X$ be a smooth variety. 
For every $i\in \mathbb Z$ such that $1\leq i\leq \dim X$, 
consider the map 
\begin{eqnarray*}
C_X^{-1}:\Omega_X^i&\to& Z_X^i/B_X^i
\end{eqnarray*}
locally defined by 
\begin{eqnarray*}
C_X^{-1}|_{\Spec R}:\Gamma(\Spec R, \Omega_X^i)&\to& \Gamma(\Spec R, Z_X^i/B_X^i)\\
da_1\wedge\cdots \wedge da_i&\mapsto&
a_1^{p-1}\cdots a_i^{p-1}da_1\wedge\cdots\wedge da_i
\end{eqnarray*}
where $\Spec R$ is an open affine subset of $X$ and $a_1,\cdots, a_i\in R.$ 
This map $C_X^{-1}$ is a well-defined $\mathcal O_X$-module isomorphism. 
We call $C_X:=(C_X^{-1})^{-1}$ the Cartier operator.  
\end{fact}

\begin{proof}
See, for example, \cite[Theorem~9.14]{EV}.
\end{proof}

\begin{rem}\label{cartier-remark}
Let $X$ be an $n$-dimensional smooth variety. 
We obtain the following exact sequences: 
\begin{enumerate}
\item{$0\to \mathcal O_X\to F_*\mathcal O_X\to B_X^1\to 0$,}
\item{$0\to Z_X^i\to F_*\Omega_X^i\to B_X^{i+1}\to 0$ for $1\leq i\leq n$, and}
\item{$0\to B_X^i\to Z_X^i\overset{C_X^i}\to \Omega_X^{i}\to 0$ for $1\leq i \leq n$.}
\end{enumerate}
By (2) for $i=n$, we obtain $Z_X^n\simeq F_*\omega_X$. 
By (3) for $i=n$, we obtain 
$$0\to B_X^n\to F_*\omega_X\overset{C_X^n}\to \omega_X\to 0.$$
\end{rem}

\begin{rem}\label{cartier-trace}
Let $X$ be an $n$-dimensional smooth projective variety. 
By Remark~\ref{trace-local}, the Cartier operator $C_X^n$ and the trace map of Frobenius are the same: 
$C_X^n=\Tr^1_X.$
\end{rem}

\begin{lem}\label{cartier-etale}
Let $\gamma:X\to Y$ be a finite \'etale morphism between smooth varieties. 
Then, for every $i$, 
$$\gamma^*B_Y^i\simeq B_{X}^i\,\,\,{\rm and}\,\,\,\gamma^*Z_Y^i\simeq Z_{X}^i.$$
\end{lem}

\begin{proof}
We may assume $X=\Spec B$ and $Y=\Spec A$. 
Let 
$$\varphi: A\to B$$ 
be the homomorphism induced by $\gamma$. 
Let 
$$F_A:A\to A\,\,\,{\rm and}\,\,\,F_B:B\to B$$
be the respective $p$-th power maps. 
Since $\varphi$ is \'etale, the following diagram is a tensor product: 
$$\begin{CD}
A @>F_A>> A\\
@VV\varphi V @VV\varphi V\\
B @>F_B>> B.
\end{CD}$$
Thus, we see that the natural $B$-module homomorphism 
\begin{eqnarray*}
\theta^i:((F_A)_*\Omega_A^i)\otimes_A B&\to& (F_B)_*(\Omega_A^i\otimes_A B)\\
(\sum_J a_Jdx_J)\otimes_A b&\mapsto& (\sum_J a_Jdx_J)\otimes_A b^p
\end{eqnarray*}
is an isomorphism where $a_J\in A$ and $dx_J:=dx_{j_1}\wedge\cdots\wedge dx_{j_i}$ for some $x_{j_l}\in A$. 
Since $\varphi$ is \'etale, the natural $B$-module homomorphism 
\begin{eqnarray*}
\rho^i:\Omega_A^i\otimes_A B&\to&\Omega_B^i\\
(\sum_J a_Jdx_J)\otimes_A b&\mapsto& \sum \varphi(a_J)bd(\varphi(x_J))
\end{eqnarray*}
is an isomorphism. 
Since $\varphi$ is flat, the isomorphisms in the lemma follow from the commutativity 
of the following diagram: 
$$\begin{CD}
(F_B)_*\Omega_B^i@>d>> (F_B)_*\Omega_B^{i+1}\\
@AA(F_B)_*\rho^iA @AA(F_B)_*\rho^{i+1}A\\
(F_B)_*(\Omega_A^i\otimes_A B)@>d>> (F_B)_*(\Omega_A^{i+1}\otimes_A B)\\
@AA\theta^iA @AA\theta^{i+1}A\\
((F_A)_*\Omega_A^i)\otimes_A B@>d>> ((F_A)_*\Omega_A^{i+1})\otimes_A B,
\end{CD}$$
which is easy to check. 
\end{proof}

We state the following vanishing result on $F$-split varieties for later use. 

\begin{prop}\label{F-split-B-vanish}
If $X$ is an $n$-dimensional $F$-split smooth projective variety and 
$A$ is an ample \Z-divisor, then 
$$H^1(X, B_X^n(A))=0.$$
\end{prop}

\begin{proof}
Consider the exact sequence 
$$0\to B_X^n\to F_*\omega_X\overset{C_X^n}\to \omega_X\to 0.$$
Then, by Proposition~\ref{F-surje}, the trace map 
$$C_X^n=\Tr_{X}^1(A):H^0(X, \omega_X(pA))\to H^0(X, \omega_X(A))$$
is surjective. 
Therefore, we obtain the exact sequence 
$$0\to H^1(X, B_X^n(A))\to H^1(X, \omega_X(pA)).$$
Since $F$-split varieties satisfy the Kodaira vanishing theorem (\cite[Proposition~2]{MR}), 
we obtain the vanishing $H^1(X, B_X^n(A))=0$.
\end{proof}

\section{The trace map of Frobenius for curves}

In this section, we calculate the trace map 
$$\Tr_{X}^e(A):H^0(X, \omega_X(p^eA))\to H^0(X, \omega_X(A))$$ 
when $X$ is a curve. 
By Remark~\ref{Tango-cex}, 
$\Tr_X^e(A)$ is not surjective in general. 
However, we show that $\Tr_X^e(A)$ is almost always a nonzero map. 

\begin{thm}\label{curve-trace-nonzero}
Let $X$ be a smooth projective curve whose genus $g(X)$ is not zero. 
Let $A$ be an ample \Z-divisor. 
Then, for every $e\in\mathbb Z_{>0}$, the trace map 
$$\Tr_{X}^e(A):H^0(X, \omega_X(p^eA))\to H^0(X, \omega_X(A))$$
is a nonzero map. 
\end{thm}

\begin{proof}
Fix $e\in\mathbb Z_{>0}$. 
Since $A$ is ample, we have $\deg A\geq 1.$ 
We consider two cases: $\deg A>1$ and $\deg A=1$. 

\setcounter{step}{0}

\begin{step}
In this step, 
we assume $\deg A>1$ and we prove the assertion. 
The following argument follows the proof of \cite[Theorem~3.3]{Schwede2}. 

Fix a point $Q\in X$. 
By Lemma~\ref{trace-diagram}, 
we obtain the following commutative diagram: 

{\Small $$\begin{CD}
H^0(X, K_X+p^eA)@>>> H^0(p^eQ, K_{p^eQ}+p^e(A-Q))@.\to H^1(X, K_X+p^e(A-Q))\\
@VV\Tr_{X}^e(A) V @VV \Psi V\\
H^0(X, K_X+A)@>\rho >> H^0(Q, K_Q+A-Q).
\end{CD}$$}
By Serre duality, we obtain the vanishing 
$$H^1(X, K_X+p^e(A-Q))=0.$$
On the other hand, $\Psi$ is surjective because the following composition 
\begin{eqnarray*}
\Tr_Q^e:H^0(Q, K_Q+p^e(A-Q))&\to &H^0(p^eQ, K_{p^eQ}+p^e(A-Q))\\
&\overset{\Psi}\to& H^0(Q, K_Q+A-Q)
\end{eqnarray*}
is surjective. 
Therefore, the composition map $\rho \circ \Tr_X^e(A)$ is surjective. 
We conclude that $\Tr_{X}^e(A)$ is a nonzero map since $H^0(Q, K_Q+A-Q)\neq 0$. 
\end{step}

\begin{step}\label{good-1-form}
In this step, we prove that 
if $\deg A=1$, then there exists a point $Q\in X$ such that 
the natural injective map 
$$H^0(X, \omega_X(p^eA-p^eQ))\to H^0(X, \omega_X(p^eA-(p^e-1)Q))$$
is not surjective.

Since $H^0(Q, L|_Q)\simeq k$ for every invertible sheaf $L$ on $X$, 
we obtain the following exact sequence 
\begin{eqnarray*}
0&\to& H^0(X, \omega_X(p^eA-p^eQ))\to H^0(X, \omega_X(p^eA-(p^e-1)Q))\to k\\
&\to& H^1(X, \omega_X(p^eA-p^eQ)).
\end{eqnarray*}
Therefore, it is sufficient to show 
$$h^1(X, \omega_X(p^eA-p^eQ))=h^0(X, -(p^eA-p^eQ))=0$$
for some point $Q\in X$. 
Note that the first equality follows from Serre duality. 
Assume the contrary, that is, 
assume that $p^eA\sim p^eQ$ for every point $Q\in X$. 
Since the genus $g(X)$ is not zero, 
there exists a nonzero $l$-torsion divisor $D$ for a prime number $l\neq p$. 
Note that $D$ is not $p^e$-torsion. 
Take the prime decomposition 
$$D=\sum m_iQ_i-\sum n_jR_j.$$
Since $\deg D=\sum m_i-\sum n_j=0$, 
we obtain the following contradiction 
\begin{eqnarray*}
p^eD&=&\sum m_ip^eQ_i-\sum n_jp^eR_j\\
&\sim&\sum m_ip^eA-\sum n_jp^eA\\
&=&(\sum m_i-\sum n_j)p^eA\\
&=&0.
\end{eqnarray*}
\end{step}

\begin{step}
In this step, 
we assume $\deg A=1$ and we prove the assertion in the theorem. 

We fix a point $Q\in X$ as in Step~\ref{good-1-form}. 
If $A\sim A'$, then the corresponding trace maps are the same by Lemma~\ref{D-linearequ}. 
Therefore, we may assume that $Q\not\in \Supp A$. 
By Step~\ref{good-1-form}, there exists an element 
$$\eta\in H^0(X, \omega_X(p^eA-(p^e-1)Q))\setminus H^0(X, \omega_X(p^eA-p^eQ)).$$
Take the local ring $(R, \mathfrak m)$ corresponding to the point $Q$. 
Note that $F^e_*R$ is a free $R$-module. 
Let $\{x\}$ be the $p$-basis. 
Then, we obtain 
$$\omega_R=\bigoplus_{0\leq i<p^e}R^{p^e}x^idx.$$
Thus, we can write 
$$\eta|_{\Spec R}=\sum_{0\leq i<p^e}f_i^{p^e}x^idx.$$
The fact that $\eta\not\in H^0(X, \omega_X(p^eA-p^eQ))$ means 
$f_i\not\in \mathfrak m$ for some $0\leq i<p^e$. 
Since $\eta\in H^0(X, \omega_X(p^eA-(p^e-1)Q))$, 
we have $f_i\in \mathfrak m$ for every $0\leq i<p^e-1$. 
Therefore, we obtain $f_{p^e-1}\not\in \mathfrak m.$
Then, we can find $c\in k^{\times}$ and $\mu\in \mathfrak m$ 
such that 
$$f_{p^e-1}=c+\mu.$$
By Remark~\ref{trace-local}, we see that $\Tr_{X}^e(A)(\eta)|_{\Spec R}\neq 0.$
\end{step}
\end{proof}

\begin{cor}\label{curve-trace-cor}
Let $X$ be a smooth projective curve. 
Let $A$ be an ample \Z-divisor. 
If $H^0(X, \omega_X(A))\neq 0$, 
then for every $e\in\mathbb Z_{>0}$, the trace map 
$$\Tr_X^e(A):H^0(X, \omega_X(p^eA))\to H^0(X, \omega_X(A))$$
is a nonzero map. 
\end{cor}

\begin{proof}
If $g(X)\geq 1$ where $g(X)$ is the genus of $X$, then 
the assertion follows from Theorem~\ref{curve-trace-nonzero}. 
Thus, we may assume that $X\simeq \mathbb P^1$. 
Since $\mathbb P^1$ is $F$-split, 
the trace map is surjective by Proposition~\ref{F-surje}. 
\end{proof}

In characteristic zero, the following result follows using the Kodaira vanishing theorem. 
In positive characteristic, we obtain the following result by  the trace map of Frobenius.

\begin{cor}\label{surface-restriction}
Let $X$ be a smooth projective surface and 
let $C$ be a smooth prime divisor on $X$. 
Let $A$ be an ample \Z-divisor on $X$. 
If $H^0(C, K_C+A|_C)\neq0$, then 
the natural restriction map 
$$H^0(X, K_X+C+A)\to H^0(C, K_C+A|_C)$$
is a nonzero map. 
\end{cor}

\begin{proof}
By Lemma~\ref{trace-diagram}, 
we obtain the following commutative diagram 
{\Small $$\begin{CD}
H^0(X, K_X+C+p^eA)@>>> H^0(C, K_{C}+p^eA|_C)@>>> H^1(X, K_X+p^eA)\\
@VV\Tr_{X, C}^e(A) V @VV \Tr_{C}^e(A|_C) V\\
H^0(X, K_X+C+A)@>>> H^0(C, K_C+A|_C).
\end{CD}$$}
We see $H^1(X, K_X+p^eA)=0$ for $e\gg 0$ by the Serre vanishing theorem. 
Thus the assertion follows from Corollary~\ref{curve-trace-cor}. 
\end{proof}

In characteristic zero, in the above situation, 
the restriction map is surjective by the Kodaira vanishing theorem. 
However, in positive characteristic, 
the restriction map is not surjective in general. 

\begin{ex}\label{cex-surjective}
There exists a smooth projective surface $X$, 
a smooth prime divisor $H$ on $X$, and 
an ample \Z-divisor $A$ 
such that 
\begin{enumerate}
\item{$|K_X+H+A|$ is basepoint free. }
\item{The natural restriction map 
$$H^0(X, K_X+H+A)\to H^0(H, K_H+A|_H)$$
is not surjective. }
\end{enumerate}
\end{ex}

\begin{proof}[Construction]
Let $X$ be a smooth projective surface and 
let $A$ be an ample \Z-divisor on $X$ such that 
$$H^1(X, K_X+A)\neq 0.$$
We can find such a surface by \cite{Raynaud}.
Take a smooth hyperplane section $H$ of $X$ such that 
$|K_X+H+A|$ is basepoint free and 
$$H^1(X, K_X+H+A)=0.$$
Consider the exact sequence 
$$0\to \mathcal O_X(K_X+A)\to\mathcal O_X(K_X+H+A)\to\mathcal O_H(K_H+A|_H)\to0.$$
Then, we obtain the following exact sequence 
$$H^0(X, K_X+H+A)\to H^0(H, K_H+A|_H)\to H^1(X, K_X+A)\to 0.$$
Since $H^1(X, K_X+A)\neq 0$, the restriction map is not surjective. 
\end{proof}

We use the following corollary in Section~8.

\begin{cor}\label{nonzero-criterion}
Let $X$ be a smooth projective surface. 
Let $L$ be a \Z-divisor on $X$ such that 
$$L=C+M,$$
where $C$ is a smooth prime divisor and $M$ is a nef and big \Z-divisor such that $M|_C$ is ample. 
If $H^0(C, K_C+M|_C)\neq 0$, then the trace map 
$$\Tr_X^e(L):H^0(X, \omega_X(p^eL))\to H^0(X, \omega_X(L))$$
is a nonzero map. 
\end{cor}

\begin{proof}
By Lemma~\ref{trace-diagram}, 
we obtain the following commutative diagram 
{\Small $$\begin{CD}
H^0(X, K_X+p^eL)@>>> H^0(p^eC, K_{p^eC}+p^eM|_C)@>>> H^1(X, K_X+p^eM)\\
@VV\Tr_{X}^e(L) V @VV \Psi V\\
H^0(X, K_X+L)@>>> H^0(C, K_C+M|_C).
\end{CD}$$}
By \cite[Theorem~2.6]{T1}, we have $H^1(X, K_X+p^eM)=0$ for $e\gg 0$. 
By Corollary~\ref{curve-trace-cor}, 
$\Psi$ is a nonzero map because the composition  
\begin{eqnarray*}
\Tr_C^e(M|_C):H^0(C, K_{C}+p^eM|_C)&\to &H^0(p^eC, K_{p^eC}+p^eM|_{p^eC})\\
&\overset{\Psi}\to& H^0(C, K_C+M|_C) 
\end{eqnarray*}
is nonzero. 
Therefore, also the trace map $\Tr_{X}^e(L)$ is a nonzero map. 
\end{proof}

\section{Surjectivity of the trace maps for surfaces ($\kappa=1$)}

In this section, we show the surjectivity of the trace map 
$$H^0(X, \omega_X(p^e(A+m(K_X+A))))\to H^0(X, \omega_X(A+m(K_X+A)))$$
when $X$ is a surface and $\kappa(X, K_X+A)=1$. 
For this, we establish the following vanishing result.

\begin{prop}\label{ruled-R1}
Let $C$ be a smooth curve. 
Let $Y:=\mathbb P^1\times C$ and 
let $\pi:Y\to C$ be the projection. 
Let $f:X\to Y$ be the blowup at one point and 
let 
$$\theta:X\overset{f}\to Y\overset{\pi}\to C.$$ 
If $A_X$ is a $\theta$-ample \Z-divisor on $X$, then 
$$R^1\theta_*(B_X^2(A_X))=0.$$
\end{prop}

\begin{proof}
\setcounter{step}{0}

\begin{step}\label{step-rational}
In this step, 
we assume $C$ is rational and we prove the assertion. 

Since the assertion is local on $C$, we may assume $C\simeq \mathbb P^1.$ 
For an arbitrary ample \Z-divisor $A_C$ on $C$, 
by the Leray spectral sequence, 
we obtain the following exact sequence  
\begin{eqnarray*}
0&\to & H^1(C, \theta_*(B_X^2(A_X))\otimes \mathcal O_C(A_C))\\
&\to & H^1(X, B_X^2(A_X+\theta^*A_C))\\
&\to & H^0(C, R^1\theta_*(B_X^2(A_X))\otimes \mathcal O_C(A_C))\to 0.\\
\end{eqnarray*}
Let $A_C$ be an ample \Z-divisor on $C$ such that 
\begin{enumerate}
\item{$A_X+\theta^*A_C$ is ample, and}
\item{$R^1\theta_*(B_X^2(A_X))\otimes \mathcal O_C(A_C)$ 
is generated by global sections. }
\end{enumerate}
Then, it is sufficient to show 
$$H^1(X, B_X^2(A_X+\theta^*A_C))=0.$$
Since $X$ is a toric variety hence $F$-split, this follows from Proposition~\ref{F-split-B-vanish}. 
\end{step}

\begin{step}\label{step-H1-R1}
In this step, we prove that 
the following assertions are equivalent. 
\begin{enumerate}
\item{$R^1\theta_*(B_X^2(A_X))=0$.}
\item{$H^1(X_c, B_X^2(A_X)|_{X_c})=0$ for every closed point $c\in C$ where $X_c=\theta^{-1}(c)$.} 
\end{enumerate}
By \cite[Theorem~12.11]{Hartshorne}, 
there exists an isomorphism 
$$R^1\theta_*(B_X^2(A_X))\otimes k(c)\simeq H^1(X_c, B_X^2(A_X)|_{X_c}).$$
By Nakayama's lemma, if 
$$R^1\theta_*(B_X^2(A_X))\otimes k(c)=0,$$
then $R^1\theta_*(B_X^2(A_X))|_{U}=0$ for some open neighborhood $c\in U\subset C$. 
\end{step}

\begin{step}
In this step, we show 
$H^1(X_c, B_X^2(A_X)|_{X_c})=0$ 
for the closed points $c\in C$ other than the one corresponding to the singular fiber. 
Fix such a closed point $c\in C$. 
We can shrink $C$ around $c\in C$. 
Thus, we may assume that $X=Y=\mathbb P^1 \times C \to C$. 
We can find a cartesian diagram
$$\begin{CD}
X':=\mathbb P^1 \times C'@<\gamma_X<< \mathbb P^1\times C=X\\
@VV\theta'V @VV\theta V\\
C'@<\gamma_C<< C
\end{CD}$$
such that $C'\simeq \mathbb A^1$ and $\gamma_C$ and $\gamma_X$ are finite \'etale morphisms. 
Note that $\gamma_X^*B_{X'}^2\simeq B_X^2$ 
by Lemma~\ref{cartier-etale}. 
We can find a $\theta'$-ample \Z-divisor $A_{X'}$ on $X'$ such that 
$(\gamma_X^*(\mathcal O_{X'}(A_{X'})))|_{X_{c}}\simeq\mathcal O_X(A_X)|_{X_c}$. 
Therefore we obtain 
\begin{eqnarray*}
H^1(X_c, B_X^2(A_X)|_{X_c})&=& H^1(X_c, \gamma_X^*B_{X'}^2(\gamma_X^*A_{X'}))\\
&=& H^1(X'_{c'}, B_{X'}^2(A_{X'}))\\
&=&0. 
\end{eqnarray*}
The last equality follows from Step~\ref{step-rational} and Step~\ref{step-H1-R1}. 
\end{step}

\begin{step}
In this step, we show 
$H^1(X_c, B_X^2(A_X)|_{X_c})=0$ 
for the closed point $c\in C$ corresponding to the singular fiber. 

We can replace $C$ by a neighborhood of $c$. 
We find a commutative diagram
$$\begin{CD}
X'@<\gamma_X<< X \\
@VVf'V @VVf V\\
Y'@<\gamma_Y<< Y\\
@VV\pi'V @VV\pi V\\
C'@<\gamma_C<< C
\end{CD}$$
where $C'\simeq \mathbb A^1$, 
each square is a fiber product, 
$\gamma_C, \gamma_Y$ and $\gamma_X$ are finite \'etale morphisms. 
Let $\theta':=\pi'\circ f'$ and $c':=\gamma_C(c).$  
Note that $\gamma_X^*B_{X'}^2\simeq B_X^2$ 
by Lemma~\ref{cartier-etale}. 
Then, $\gamma_X|_{X_c}:X_c\to X'_{c'}$ 
is an isomorphism. 

We show that there exists a $\theta'$-ample \Z-divisor $A_{X'}$ on $X'$ such that 
$(\gamma_X^*(\mathcal O_{X'}(A_{X'})))|_{X_{c}}\simeq\mathcal O_X(A_X)|_{X_c}$. 
The fiber $X_c \simeq X'_{c'}=:D$ is a reduced simple normal crossing divisor 
$D_1 \cup D_2$, with $D_i \simeq \mathbb P^1$. 
We can assume that $D_1$ is the $f'$-exceptional curve and 
$D_2$ is the proper transform of a fiber of $Y' \to C'$. 
We consider the following exact sequence: 
$$1 \to \MO_D^{\times} \to \MO_{D_1}^{\times} \times \MO_{D_2}^{\times} \to \MO_{D_1 \cap D_2}^{\times} \to 1.$$
Since the intersection $D_1 \cap D_2$ is one point, 
$$H^0(X, \MO_{D_1}^{\times} \times \MO_{D_2}^{\times}) \to H^0(X, \MO_{D_1 \cap D_2}^{\times})$$
is surjective. 
Thus, we obtain the following group isomorphism 
$${\rm Pic} D \xrightarrow{\simeq}  {\rm Pic} D_1 \times{\rm Pic} D_2,\,\,\, L \mapsto (L|_{Y_1}, L|_{Y_2}).$$
Therefore it suffices to show that, 
for given positive integers $a_1, a_2 \in \mathbb Z_{>0}$, 
there exists an invertible sheaf $A_{X'}$ on $X'$ such that $A_{X'}|_{D_i}\simeq \MO_{D_i}(a_i)$ for each $i=1, 2$. 
Consider a section $S'$ of $Y' \to C'$ passing through the blown-up point and 
its proper transform $T' \subset X'$. 
Then, $\MO_{X'}(T')|_{Y_1}\simeq \MO_{D_1}(1)$ and $\MO_{X'}(T')|_{Y_2}\simeq \MO_{D_2}$. 
Since 
$$\MO_{X'}(nT'+mD_1)|_{D_1} \simeq \MO_{D_1}(n-m),\,\,\, \MO_{X'}(nT'+mD_1)|_{D_2} \simeq \MO_{D_1}(m),$$
we are done. 

\medskip

Thus we obtain 
\begin{eqnarray*}
H^1(X_c, B_X^2(A_X))&=& H^1(X_c, \gamma_X^*B_{X'}^2(\gamma_X^*A_{X'}))\\
&=& H^1(X'_{c'}, B_{X'}^2(A_{X'}))\\
&=&0. 
\end{eqnarray*}
The last equality follows from Step~\ref{step-rational} and Step~\ref{step-H1-R1}. 
\end{step}

The assertion in the proposition follows from Step~\ref{step-H1-R1}, Step~3 and Step~4. 
\end{proof}

Let us prove the main theorem in this section. 

\begin{thm}\label{trace-surface-1}
Let $X$ be a smooth projective surface and 
let $A$ be an ample \Z-divisor on $X$ such that $\kappa(X, K_X+A)=1$ and $K_X+A$ is nef. 
Then there exists $m_1\in\mathbb Z_{>0}$ such that
the trace map 
$$\Tr_X^e(A+m(K_X+A)):$$ 
$$H^0(X, K_X+p^e(A+m(K_X+A)))\to H^0(X, K_X+(A+m(K_X+A)))$$
is surjective for every integer $m\geq m_1$ and for every $e\in\mathbb Z_{>0}$. 
\end{thm}

\begin{proof}
\setcounter{step}{0}
\begin{step}
We see that 
$K_X+A$ is semi-ample by \cite[The second theorem in Introduction]{Fujita3}. 
\end{step}

\begin{step}
In this step we prove that, for some $n_0\in\mathbb Z_{>0}$, 
the complete linear system 
$$\Phi_{|n_0(K_X+A)|}=\theta:X\to C$$
gives a ruled surface structure, that is, 
$\theta$ is a projective morphism to a smooth projective curve 
such that $\theta_*\mathcal O_X=\mathcal O_C$ and a general fiber is $\mathbb P^1$. 

We can find $n_0\in\mathbb Z_{>0}$ such that 
$\Phi_{|n_0(K_X+A)|}=\theta:X\to C$ is a projective morphism to a smooth projective curve 
such that $\theta_*\mathcal O_X=\mathcal O_C$. 
By \cite[Corollary~7.3]{Badescu}, general fibers are integral. 
Since a general fiber $F$ satisfies 
$$0=(K_X+A)\cdot F>(K_X+F)\cdot F,$$ 
we see $F\simeq \mathbb P^1.$ 
Thus, $\theta$ gives a ruled surface structure. 
\end{step}

\begin{step}\label{step-R1}
In this step we prove that it is sufficient to show 
$$R^1\theta_*(B_X^2(A'))=0$$
for every ample \Z-divisor $A'$. 

By Remark~\ref{cartier-remark} and Remark~\ref{cartier-trace}, 
we obtain the following exact sequence 
$$0\to B_X^2\to F_*\omega_X\overset{\Tr_X^1}\to \omega_X\to 0.$$
By Lemma~\ref{e-vs-e+1}, 
it suffices to find $m_0\in \mathbb Z_{>0}$ 
such that 
$$H^1(X, B_X^2(p^d(A+m(K_X+A))))=0$$
for every $d \geq 0$ and $m\geq m_0$. 
By the Frujita vanishing theorem, 
we can find $d_0\in\mathbb Z_{>0}$ such that 
if $d>d_0$, then 
$$H^1(X, B_X^2(p^d(A+m(K_X+A))))=0$$
for every $m\geq 0$. 
Therefore, we fix an integer $0\leq d \leq d_0$, and 
it is enough to find $n_d\in\mathbb Z_{>0}$, depending on $d$, such that  
$$H^1(X, B_X^2(p^d(A+m(K_X+A))))=0$$
for every $m\geq n_d$. 
We can write $K_X+A=\theta^*H$ where $H$ is an ample \Z-divisor on $C$. 
By the Leray spectral sequence, we obtain 
\begin{eqnarray*}
0&\to&H^1(C, \theta_*(B_X^2(p^dA))\otimes_{\MO_C} \MO_C(p^dmH))=0\\
&\to&H^1(X, B_X^2(p^d(A+m(K_X+A))))\\
&\to&H^0(C, R^1\theta_*(B_X^2(p^dA))\otimes_{\MO_C} \MO_C(p^dmH))\to 0.
\end{eqnarray*}
If $m\gg 0$, then the first term vanishes by the Serre vanishing theorem. 
Thus, it is sufficient to show that 
$R^1\theta_*(B_X^2(A'))=0$ for every ample \Z-divisor $A'$. 
\end{step}

\begin{step}
In this step, we prove the assertion in the theorem. 
By Step~\ref{step-R1}, it is sufficient to show 
$$R^1\theta_*(B_X^2(A'))=0$$
for every ample \Z-divisor $A'$. 
 
Since $X \to C$ is a ruled surface structure, 
after contracting $(-1)$-curves in fibers, 
we obtain morphisms 
$$\theta:X\overset{f}\to Y\overset{\pi}\to C$$
where $f$ is a birational morphism to a smooth projective surface $Y$. 
Note that every fiber of $\pi$ is irreducible, otherwise we can find a $(-1)$-curve in a reducible fiber. 
By the adjunction formula, every fiber of $\pi$ is $\mathbb P^1$. 
Thus, $\pi$ is a $\mathbb P^1$-bundle structure (cf. \cite[Corollary~11.11]{Badescu}). 
Since the problem is local on $C$, we may assume that 
$\theta$ has only one singular fiber and that 
$Y=\mathbb P^1\times C.$ 
Let $F_s$ be the singular fiber. 
We see that 
$$0=(K_X+A)\cdot F_s=-2+A\cdot F_s.$$
Thus, $F_s$ has at most two irreducible components. 
This implies that $f$ is the blowup at one point or an isomorphism. 
Then, the equation $R^1\theta_*(B_X^2(A'))=0$ follows from Proposition~\ref{ruled-R1}. 
\end{step}
\end{proof}

\section{Surjectivity of the trace maps for surfaces ($\kappa=2$)}

In this section, we show the surjectivity of the trace map 
$$H^0(X, \omega_X(p^e(A+m(K_X+A))))\to H^0(X, \omega_X(A+m(K_X+A)))$$
when $X$ is a surface and $\kappa(X, K_X+A)=2$. 
Let us recall a lemma on global generation. 

\begin{lem}\label{uniform-generation}
Let $X$ be a smooth projective variety. 
Let $A$ be an ample \Z-divisor and let $G$ be a coherent sheaf. 
Then, there exists $n_0\in\mathbb Z_{>0}$, depending only on $A$ and $G$, 
such that 
$$G(n_0A+N)$$ 
is generated by global sections for every nef \Z-divisor $N$. 
\end{lem}

\begin{proof}
The assertion immediately follows from 
Castelnuovo--Mumford regularity (\cite[Theorem~1.8.5]{Lazarsfeld}) and 
the Fujita vanishing theorem (\cite[Theorem~(1)]{Fujita1}, \cite[Section~5]{Fujita2}). 
\end{proof}

To prove the surjectivity, we establish the following vanishing result. 

\begin{prop}\label{birat-B-vanish}
Let $h:X\to Z$ be a birational morphism between 
smooth projective surfaces. 
Let $A_X$ be an ample \Z-divisor on $X$ and 
let $A_Z$ be an ample \Z-divisor on $Z$. 
Then there exists $m_0\in \mathbb Z_{>0}$ such that 
$$H^1(X, B_X^2(A_X+h^*(m_0A_Z+N_Z)))=0$$
for every nef \Z-divisor $N_Z$ on $Z$.
\end{prop}

\begin{proof}
The birational morphism $h$ is a composition of $n$ point blowups. 
We prove the assertion by induction on $n$.

\setcounter{step}{0}

\begin{step}\label{step-n=0}
If $n=0$, then the assertion follows from the Fujita vanishing theorem 
(\cite[Theorem~(1)]{Fujita1}, \cite[Section~5]{Fujita2}). 
Thus, we may assume that $n>0$ and the assertion holds for $n-1$. 
\end{step}

\begin{step}\label{step-one-point}
In this step, we prove that we may assume that $h(\Ex(h))$ is one point. 

Assume that the assertion holds when $h(\Ex(h))$ is one point. 
The Leray spectral sequence induces a short exact sequence 
\begin{eqnarray*}
0&\to& H^1(Z, h_*(B_X^2(A_X+h^*(m_0A_Z+N_Z))))=0\\
&\to& H^1(X, B_X^2(A_X+h^*(m_0A_Z+N_Z)))\\
&\to& H^0(Z, R^1h_*(B_X^2(A_X+h^*(m_0A_Z+N_Z))))\to 0\\
\end{eqnarray*}
where the equation 
$H^1(Z, h_*(B_X^2(A_X+h^*(m_0A_Z+N_Z))))=0$ follows from the Fujita vanishing theorem 
(\cite[Theorem~(1)]{Fujita1}, \cite[Section~5]{Fujita2}). 
By Lemma~\ref{uniform-generation}, the following assertions are equivalent. 
\begin{itemize}
\item{$H^1(X, B_X^2(A_X+h^*(m_0A_Z+N_Z)))=0$.}
\item{$R^1h_*(B_X^2(A_X))=0.$}
\end{itemize}
Set $h(\Ex(h))=\{z_0, z_1, \cdots, z_m\}$ and 
$Z_0:=Z\setminus \{z_1, \cdots, z_m\}$. 
We only show $R^1h_*(B_X^2(A_X))|_{Z_0}=0$. 
Let $X'$ be the smooth projective surface 
obtained by patching $Z\setminus \{z_0\}$ and $X\setminus h^{-1}(\{z_1, \cdots, z_m\})$. 
We obtain a birational morphism $h':X' \to Z$ of smooth projective surfaces 
such that $h'(\Ex(h'))$ is equal to $\{z_0\}$ and $h'|_{h'^{-1}(Z_0)}=h|_{h^{-1}(Z_0)}$. 
Let $A_{X'}^{(0)}$ be the closure of $A_X|_{h^{-1}(Z_0)}$. 
Since $A_{X'}^{(0)}$ and $A_X|_{h^{-1}(Z)}$ are the same around $\Ex(h')$, 
we see that $A_{X'}^{(0)}$ is $h'$-ample. 
We fix $n_0\gg 0$ such that $A_{X'}:=A_{X'}^{(0)}+n_0h'^*A_Z$ is ample. 
By our assumption, we obtain 
$$H^1(X', B_{X'}^2(A_{X'}+h'^*(m_0A_Z+N_Z)))=0.$$
By the above argument using the Leray spectral sequence, this is equivalent to 
$$R^1h'_*(B_X^2(A_{X'}))=0.$$
This implies 
$$0=R^1h'_*(B_X^2(A_{X'}^{(0)}))|_{Z_0}=R^1h_*(B_X^2(A_{X}))|_{Z_0}.$$
We are done. 
\end{step}

\medskip
From now on, we assume that $h(\Ex(h))$ is one point.

\begin{step}\label{step-s=0}
We consider the factorization 
$$h:X\overset{f}\to Y\overset{g}\to Z,$$
where $g$ is the blowup of $Z$ at a point. 
Let $E_Y$ be the $g$-exceptional curve. 
Note that $E_Y^2=-1$. 
We see 
$$g^*A_Z-\frac{1}{l} E_Y$$
is an ample \Q-divisor for every large integer $l \gg 0.$ 
Thus, by replacing $A_Z$ with a multiple, we may assume that 
$$A_Y:=g^*A_Z-E_Y$$ 
is an ample \Z-divisor. 
In particular, we obtain 
$$h^*A_Z=f^*A_Y+f^*E_Y.$$
By the induction hypothesis, 
there exists $m_1\in\mathbb Z_{>0}$ such that 
$$H^1(X, B_X^2(A_X+f^*(m_1A_Y+N_Y)))=0$$
for every nef \Z-divisor $N_Y$ on $Y$. 
We have 
$$m_1h^*A_Z=m_1f^*A_Y+m_1f^*E_Y.$$
\end{step}

\begin{step}\label{step-construct-E}
Let $E_1,\cdots, E_n \subset X$ be the $h$-exceptional curves 
where $E_1$ is the proper transform of $E_Y$. 
In this step 
we construct a sequence of \Z-divisors 
$$0=:E(0)\leq E(1)\leq E(2)\leq\cdots\leq E(R-1)\leq E(R):=f^*E_Y$$
such that 
\begin{enumerate}
\item[(a)]{For every $0\leq r\leq R-1$, $E(r+1)-E(r)=E_j$ for some $1\leq j\leq n$.}
\item[(b)]{$E(r)\cdot E_i\geq -1$ for every $0\leq r\leq R$ and for every $1\leq i\leq n$. }
\end{enumerate}

We consider a decomposition into one point blowups:
$$f:X=:X^{(n)}\overset{f^{(n)}}\to\cdots\overset{f^{(3)}}\to X^{(2)}\overset{f^{(2)}}\to X^{(1)}:=Y.$$
We may assume that, for every $2\leq j\leq n$, 
$E_j\subset X$ is the proper transform of the $f^{(j)}$-exceptional curve. 
For $1\leq j\leq i\leq n$, let $E_j^{(i)}\subset X^{(i)}$ be the image of $E_j$ 
(e.g. the $f^{(i)}$-exceptional curve is $E_i^{(i)} \subset X^{(i)}$). 
Let $f^{(i+1)}(\Ex(f^{(i+1)}))=:P^{(i)}\in X^{(i)}$ and denote by $g^{(i)}$ 
the induced map 
$$g^{(i)}:X^{(i)}\to Z.$$ 
Note that $P^{(i)}\in \Ex(g^{(i)})$. 
Since $\Supp(\Ex(g^{(i)}))$ is simple normal crossing, 
there are two cases: 
\begin{enumerate}
\item{$P^{(i)}\in E_j^{(i)}$ for some $j$ and $P^{(i)}\not\in E_{j'}^{(i)}$ for every $j'\neq j$.}
\item{$P^{(i)}\in E_j^{(i)}\cap E_{j'}^{(i)}$ for some $j\neq j'$ and $P^{(i)}\not\in E_{j''}^{(i)}$ for every $j''\neq j, j'$.}
\end{enumerate}
For $1\leq i\leq n$, 
we construct a finite sequence $({\rm Seq})_i$ of prime divisors on $X$ inductively as follows.
Every member of $({\rm Seq})_i$ is equal to $E_j$ for some $j$. 
Let 
\begin{eqnarray*}
({\rm Seq})_1:=(E_1).
\end{eqnarray*}
Assume we obtain $({\rm Seq})_i$. 
We construct $({\rm Seq})_{i+1}$ as follows. 
There are two cases (1) and (2) as above. 
Assume (1), that is, $P_i\in E_j^{(i)}$ and $P_i\not\in E_{j'}^{(i)}$ for every $j'\neq j$. 
If 
$$({\rm Seq})_i=(\cdots, E_j,\cdots, E_{j'},\cdots),$$
then we define $({\rm Seq})_{i+1}$ by 
$$({\rm Seq})_{i+1}:=(\cdots, E_{i+1}, E_j,\cdots, E_{j'},\cdots).$$
In other words, 
we add $E_{i+1}$ only in front of $E_j$ (for each appearance of $E_j$). 
Assume (2), that is, $P^{(i)}\in E_j^{(i)}\cap E_{j'}^{(i)}$ for some $j\neq j'$ and 
$P_i\not\in E_{j''}^{(i)}$ for every $j''\neq j, j'$.
If 
$$({\rm Seq})_i=(\cdots, E_j,\cdots, E_{j'},\cdots, E_{j''},\cdots),$$
then we define $({\rm Seq})_{i+1}$ by 
$$({\rm Seq})_{i+1}:=(\cdots, E_{i+1}, E_j,\cdots, E_{i+1}, E_{j'},\cdots, E_{j''},\cdots).$$
In other words, 
we add $E_{i+1}$ only in front of $E_j$ and $E_{j'}$ (for each appearance of $E_j$ or $E_{j'}$). 
We obtain a finite sequence $({\rm Seq})_i$ for $1\leq i\leq n.$
Let 
$$({\rm Seq})_n=(E_{a(1)}, E_{a(2)}, E_{a(3)},\cdots, E_{a(R)})$$
where $a(l)\in\{1, \cdots, n\}.$ 
We define a finite sequence $({\rm SEQ})$ by 
\begin{eqnarray*}
({\rm SEQ})&=&
(E_{a(1)}, E_{a(1)}+E_{a(2)}, E_{a(1)}+E_{a(2)}+E_{a(3)},\cdots)\\
&=:&(E(1), E(2), E(3),\cdots, E(R)).
\end{eqnarray*} 
It suffices to show $E(r)\cdot E_j\geq -1$ for every $j$ and $E(R)=f^*E_Y$. 
By our construction, we can check that $E(R)=f^*E_Y$. 
We prove $E(r)\cdot E_j\geq -1$ for every $j$ and $r$ by induction on $n$. 

We show that $E(r)\cdot E_n\geq -1$ for every $1\leq r\leq R$. 
We consider the behavior of the sequence 
$$E(1)\cdot E_n, E(2) \cdot E_n, \cdots.$$
If $E(r)\cdot E_n> E(r+1)\cdot E_n$, then $E_{a(r+1)}=E_n$. 
Thus, we consider the subset $K:=\{k_1, \cdots, k_{\nu}\} \subset \{1,\cdots, R\}$ 
with $k_1<k_2< \cdots <k_{\nu}$ 
such that $E_{a(k_1)}=\cdots=E_{a(k_{\nu})}=E_n$ 
and that $E_{a(r')}\neq E_n$ for every $r' \in \{1,\cdots, R\}\setminus K$: 
{\scriptsize $$({\rm Seq})_n=(\cdots, E_{a(k_1)}=E_n, E_{a(k_1+1)}, \cdots, E_{a(k_2)}=E_n, E_{a(k_2+1)}, \cdots, 
E_{a(k_{\nu})}=E_n, E_{a(k_{\nu}+1)}, 
\cdots).$$} 
By the construction, we see $E_n \cap E_{a(k+1)}\neq\emptyset$ for every $k\in K$. 
Therefore, we obtain 
$$E(k_1) \cdot E_n \geq -1,\,\, E(k_1+1)\cdot E_n=(E(k_1)+E_{a(k_1+1)})\cdot E_n \geq -1+E_{a(k_1+1)}\cdot E_n \geq 0$$
$$E(k_2) \cdot E_n \geq -1,\,\, E(k_2+1)\cdot E_n=(E(k_2)+E_{a(k_2+1)})\cdot E_n \geq -1+E_{a(k_2+1)}\cdot E_n \geq 0$$
$$\cdots.$$
Thus we see $E(r) \cdot E_n \geq -1$ for every $1\leq r \leq R$. 

We consider the corresponding sequences $({\rm Seq})_{n-1}^{(n-1)}$ and $({\rm SEQ})^{(n-1)}$ 
on $X^{(n-1)}$, that is, 
\begin{eqnarray*}
({\rm Seq})^{(n-1)}_n&:=&(E_{a(1)}^{(n-1)}, E_{a(2)}^{(n-1)}, E_{a(3)}^{(n-1)},\cdots, E_{a(R)}^{(n-1)})\\
({\rm SEQ})^{(n-1)}&:=&(E_{a(1)}^{(n-1)}, E_{a(1)}^{(n-1)}+E_{a(2)}^{(n-1)}, E_{a(1)}^{(n-1)}+E_{a(2)}^{(n-1)}+E_{a(3)}^{(n-1)},\cdots)\\
&=:&(E^{(n-1)}(1), E^{(n-1)}(2), E^{(n-1)}(3), \cdots, E^{(n-1)}(R)), 
\end{eqnarray*}
where we set $E_{n}^{(n-1)}:=0$. 
By the induction hypothesis, 
we obtain 
$$E^{(n-1)}(r)\cdot E_i^{(n-1)} \geq -1$$
for every $1\leq r\leq R$ and $1\leq i \leq n$. 
By our construction, we see 
$$(f^{(n)})^*(E^{(n-1)}(r))=E(r)$$
for every $r \in \{1, \cdots, R\}\setminus K$. 
Therefore, for $r \in \{1, \cdots, R\}\setminus K$ and $1\leq i \leq n$, 
we obtain 
$$E(r) \cdot E_i=(f^{(n)})^*(E^{(n-1)}(r))\cdot E_i=E^{(n-1)}(r)\cdot E_i^{(n-1)} \geq -1.$$
Thus, it suffices to show that 
$$E(r) \cdot E_i \geq -1$$
for $r \in K$ and $1\leq i \leq n-1$. 
This follows from 
$$E(r) \cdot E_i =(E(r-1)+E_{a(r)}) \cdot E_i=(E(r-1)+E_{n}) \cdot E_i \geq E(r-1)\cdot E_i \geq -1.$$
We are done. 
\end{step}

\begin{step}\label{step-construct-D}
In this step we construct a sequence of \Z-divisors 
$$0=:D(0)\leq D(1)\leq D(2)\leq\cdots\leq D(S-1)\leq D(S):=m_1f^*E_Y$$
such that 
\begin{enumerate}
\item[(a)]{For every $0\leq s\leq S-1$, $D(s+1)-D(s)=E_j$ for some $j$. }
\item[(b)]{$D(s)+A_X+m_1f^*A_Y$ is nef for every $0\leq s\leq S$. }
\end{enumerate}

We define the sequence $\{D(s)\}_{s=0}^S$ by 
\begin{eqnarray*}
&&E(0),\\
&&E(1),E(2),\cdots, E(R), \\
&&E(R)+E(1),E(R)+E(2),\cdots, 2E(R),\\
&&2E(R)+E(1),2E(R)+E(2),\cdots, 3E(R),\\
&&\cdots\\
&&(m_1-1)E(R)+E(1),(m_1l-1)E(R)+E(2),\cdots, m_1E(R).
\end{eqnarray*}
Then, the sequence $\{D(s)\}_{s=0}^S$ satisfies (a). 
We now show (b). 
For every $0\leq s\leq S$, we can write 
$$D(s)+A_X+m_1f^*A_Y=E(r)+tf^*E_Y+A_X+m_1f^*A_Y.$$
for some $0\leq r\leq R$ and some $0\leq t\leq m_1-1.$ 
To show that this divisor is nef, 
it is sufficient to show that 
$$(E(r)+tf^*E_Y+A_X+m_1f^*A_Y)\cdot E_j\geq 0$$
for every $1\leq j\leq n.$ 
By Step~\ref{step-construct-E}, 
for every $2\leq j\leq n$, we obtain
$$(E(r)+tf^*E_Y+A_X+m_1f^*A_Y)\cdot E_j=(E(r)+A_X)\cdot E_j\geq 0.$$
On the other hand, when $j=1$, 
we have 
\begin{eqnarray*}
(E(r)+tf^*E_Y+A_X+m_1f^*A_Y)\cdot E_1
&\geq & (tf^*E_Y+m_1f^*A_Y)\cdot E_1\\
&=&(tE_Y+m_1A_Y)\cdot E_Y\\
&\geq &(m_1E_Y+m_1A_Y)\cdot E_Y\\
&=&0.
\end{eqnarray*}
\end{step}

\begin{step}\label{step-trace}
For a \Z-divisor $D$ on $X$ and for a curve $E\simeq\mathbb P^1$ in $X$, 
by Lemma~\ref{trace-diagram}, 
we obtain the following diagram:  
{\tiny $$\begin{CD}
0@>>> H^0(X, \omega_X(pD))@>>> H^0(X, \omega_X(E+pD))@>>> H^0(E, \omega_E(pD))@>>> 
H^1(X, \omega_X(pD))\\
@. @VV\alpha:=\Tr_X(D) V @VV\beta:=\Tr_{X, E}(D) V @VV\gamma:=\Tr_E(D) V\\
0@>>>H^0(X, \omega_X(D))@>>> H^0(X, \omega_X(E+D))@>>> H^0(E, \omega_E(D))
\end{CD}$$}
where the horizontal sequences are exact and 
the vertical arrows are the trace maps. 
Then the following assertions hold: 
\begin{enumerate}
\item{$\gamma$ is surjective.}
\item{If $H^1(X, \omega_X(pD))=0$ and $\alpha$ is surjective, then also $\beta$ is surjective. }
\item{If $\beta$ is surjective, then also the trace map 
$$\Tr_X(E+D):H^0(X, \omega_X(p(E+D)))\to H^0(X, \omega_X(E+D))$$
is surjective.}
\end{enumerate}
The assertion in (1) holds because $E\simeq \mathbb P^1$ is $F$-split (Proposition~\ref{F-surje}). 
We deduce (2) from the snake lemma. 
The assertion in (3) follows from Remark~\ref{pE-vs-E}. 
\end{step}

\begin{step}\label{step-final}
Let $m_2\in \mathbb Z_{>0}$ such that 
$$H^1(X, \omega_X(m_2h^*A_Z+N_X))=0$$
for every nef \Z-divisor $N_X$ on $X$. 
Note that, since $h^*A_Z$ is nef and big, 
we can find such an integer $m_2$ by \cite[Theorem~2.6]{T1}. 
Let $m_0:=m_1+m_2$ and 
fix a nef \Z-divisor $N_Z$ on $Z$. 

We would like to apply the diagram in Step~\ref{step-trace} for 
$$D=D(s)+A_X+m_1f^*A_Y+m_2h^*A_Z+h^*N_Z,\,\,\,\,E=D(s+1)-D(s)$$
where $0\leq s\leq S-1$. 
Note that, by Step~\ref{step-construct-D}, this divisor $D$ is nef. 
By Step~\ref{step-s=0}, 
$\alpha=\Tr_X(D)$ in Step~\ref{step-trace} is surjective for 
$$D=D(0)+A_X+m_1f^*A_Y+m_2h^*A_Z+h^*N_Z.$$ 
We have 
$$H^1(X, \omega_X(p(D(s)+A_X+m_1f^*A_Y+m_2h^*A_Z+h^*N_Z)))=0$$
by the choice of $m_2$. 
Therefore, 
by Step~\ref{step-construct-D} and Step~\ref{step-trace}, 
we obtain the surjection 
$$\Tr_X(D):H^0(X, \omega_X(pD))\to H^0(X, \omega_X(D))$$
for 
\begin{eqnarray*}
D&=&D(S)+A_X+m_1f^*A_Y+m_2h^*A_Z+h^*N_Z\\
&=&m_1f^*E_Y+A_X+m_1f^*A_Y+m_2h^*A_Z+h^*N_Z\\
&=&A_X+m_1h^*A_Z+m_2h^*A_Z+h^*N_Z\\
&=&A_X+(m_1+m_2)h^*A_Z+h^*N_Z\\
&=&A_X+h^*(m_0A_Z+N_Z).
\end{eqnarray*}
Thus, the assertion in the proposition follows from 
$$H^1(X, \omega_X(p(A_X+h^*(m_0A_Z+N_Z))))=0.$$
\end{step}
\end{proof}

\begin{prop}\label{birat-surje}
Let $h:X\to Z$ be a birational morphism between 
smooth projective surfaces. 
Let $A_X$ be an ample \Z-divisor on $X$ and 
let $A_Z$ be an ample \Z-divisor on $Z$. 
Then, there exists $m_1\in \mathbb Z_{>0}$ such that 
the trace map $\Tr^e_X(A_X+h^*(m_1A_Z+N_Z))$
{\Small 
$$H^0(X, \omega_X(p^e(A_X+h^*(m_1A_Z+N_Z))))\to H^0(X, \omega_X(A_X+h^*(m_1A_Z+N_Z)))$$}
is surjective for every $e\in\mathbb Z_{>0}$ and for every nef \Z-divisor $N_Z$ on $Z$.
\end{prop}

\begin{proof}
For $m\in\mathbb Z_{>0}$ and a nef \Z-divisor $N_Z$ on $Z$, 
set 
$$D(m, N_Z):=A_X+h^*(mA_Z+N_Z).$$ 
By Lemma~\ref{e-vs-e+1}, we obtain  
$$\Tr_{X}^{d+1}(D(m, N_Z))=\Tr_{X}^d(D(m, N_Z))\circ F_*^d(\Tr_{X}(p^d D(m, N_Z))).$$
Thus, it suffices to find $m_1\in\mathbb Z_{>0}$ 
such that $\Tr_{X}(p^d D(m_1, N_Z))$ is surjective 
for every $d\in \mathbb Z_{\geq 0}$ and nef \Z-divisor $N_Z$ on $Z$. 
By the Fujita vanishing theorem, 
we can find $d_0 \in\mathbb Z_{>0}$ such that 
$\Tr_{X}(p^d D(m_1, N_Z))$ is surjective 
for every $d>d_0$, $m_1\in\mathbb Z_{>0}$ and nef \Z-divisor $N_Z$ on $Z$. 
By Proposition~\ref{birat-B-vanish}, 
for every $0\leq i \leq d_0$, 
we can find $n_i \in \mathbb Z_{>0}$ such that 
$\Tr_{X}(p^iD(m, N_Z))$ is surjective for every $m \geq n_i$ and nef \Z-divisor $N_Z$ on $Z$. 
Therefore, for $m_1:=\max_{1\leq i\leq d_0}\{n_i\}$, 
the trace map $\Tr_{X}(p^dD(m_1, N_Z))$ is surjective 
for every $d\in\mathbb Z_{\geq 0}$ and nef \Z-divisor $N_Z$ on $Z$. 
\end{proof}

Let us prove the main theorem in this section. 

\begin{thm}\label{trace-surface-2}
Let $X$ be a smooth projective surface. 
Let $A$ be an ample \Z-divisor on $X$ such that 
$K_X+A$ is nef and big. 
Then there exists $m_1\in\mathbb Z_{>0}$ such that 
the trace map $\Tr_X^e(A+m(K_X+A))$
$$H^0(X, \omega_X(p^e(A+m(K_X+A))))\to H^0(X, \omega_X(A+m(K_X+A)))$$
is surjective for every $m\geq m_1$ and for every $e\in\mathbb Z_{>0}$. 
\end{thm}

\begin{proof}
By Proposition~\ref{birat-surje}, 
it is sufficient to prove that 
there exists a birational morphism 
$$h:X\to Z$$
to a smooth projective surface $Z$ such that 
$K_X+A$ is the pull-back of an ample \Z-divisor on $Z$. 
If $K_X+A$ is ample, then there is nothing to show. 
We may assume that $K_X+A$ is not ample. 
Then, by the Nakai--Moishezon criterion, 
we can find a curve $E$ such that $(K_X+A)\cdot E=0$. 
This implies $K_X\cdot E<0$. 
Moreover, since $K_X+A$ is big, the equation $(K_X+A)\cdot E=0$ 
implies $E^2<0$. 
Therefore, $E$ is a $(-1)$-curve. 
Let $g:X \to Y$ be the contraction of $E$. 
Set $A_Y:=g_*A$. 
Then, we see that $A_Y$ is ample and $K_X+A=g^*(K_Y+A_Y)$.  
We can apply the same argument to $Y$ and $A_Y$, that is, 
$K_Y+A_Y$ is ample or we can find a $(-1)$-curve $E_Y$ on $Y$ with $(K_Y+A_Y)\cdot E_Y=0$. 
Since $\rho(Y)=\rho(X)-1$, this procedure terminates. 
Thus, we obtain a birational morphism $h:X\to Z$ to 
a smooth projective surface such that $K_X+A=h^*(K_Z+h_*A)$ where $K_Z+h_*A$ is ample. 
\end{proof}

\section{Main theorem for threefolds}

In this section, we prove the main theorem for threefolds. 
Let us summarize the results on the trace map obtained in the previous sections. 

\begin{thm}\label{trace-surface}
Let $X$ be a smooth projective surface. 
Let $A$ be an ample \Z-divisor on $X$ such that $K_X+A$ is nef and $\kappa(X, K_X+A)\neq 0$. 
Then there exists $m_1\in\mathbb Z_{>0}$ such that 
the trace map 
$$\Tr_X^e(A+m(K_X+A)):$$
$$H^0(X, K_X+p^e(A+m(K_X+A)))\to H^0(X, K_X+(A+m(K_X+A)))$$
is surjective for every $m\geq m_1$ and for every $e\in\mathbb Z_{>0}$. 
\end{thm}

\begin{proof}
If $\kappa(X, K_X+A)=-\infty$, then there is nothing to show. 
Thus, we may assume $\kappa(X, K_X+A)\geq 1$. 
Then, the assertion follows from 
Theorem~\ref{trace-surface-1} and Theorem~\ref{trace-surface-2}
\end{proof}

\begin{rem}
In the above situation, 
one can show $\kappa(X, K_X+A)\neq -\infty$ using the abundance theorem obtained in \cite[Theorem~1.4]{Fujita3}. 
Indeed, by Bertini's theorem, 
we can find an effective \Q-divisor $D$ such that $\llcorner D\lrcorner=0$ and $A\sim_{\mathbb Q}D.$
\end{rem}

Let us prove the main theorem. 

\begin{thm}\label{threefold-extension}
Let $X$ be a smooth projective threefold. 
Let $S$ be a smooth prime divisor on $X$ and 
let $A$ be an ample \Z-divisor on $X$ such that 
\begin{enumerate}
\item{$K_X+S+A$ is nef, and }
\item{$\kappa(S, K_S+A|_S)\neq 0.$}
\end{enumerate}
Then there exists $m_0\in \mathbb Z_{>0}$ such that, for every integer $m\geq m_0$, 
the natural restriction map 
$$H^0(X, m(K_X+S+A))\to H^0(S, m(K_S+A|_S))$$
is surjective.
\end{thm}

\begin{proof}
Let $L:=K_X+S+A$. 
By Lemma~\ref{trace-diagram}, 
we obtain the following commutative diagram 
{\Small 
$$\begin{CD}
H^0(X, \omega_X(S+p^eA+p^emL))@>>> H^0(S, \omega_S(p^eA|_S+mp^eL|_S))@>>> 
H^1(X, \omega_X(p^eA+p^emL))\\
@VV\Tr_{X,S}^e(A+mL) V @VV\Tr_S^e(A|_S+mL|_S) V\\
H^0(X, \omega_X(S+A+mL))@>>> H^0(S, \omega_S(A|_S+mL|_S)).
\end{CD}$$}
By (2) and Theorem~\ref{trace-surface}, 
we can find $m_0\in\mathbb Z_{>0}$ such that 
the trace map $\Tr_S^e(A|_S+mL|_S)$
$$H^0(S, K_S+p^eA+p^emL)\to H^0(S, K_S+A|_S+mL|_S)$$
is surjective for every $m\geq m_0$ and $e\in\mathbb Z_{>0}$. 
Fix an integer $m\geq m_0$. 
By the Serre vanishing theorem, we have 
$$H^1(X, \omega_X(p^eA+p^emL))=0$$
for $e\gg 0$. 
Therefore, the natural restriction map 
\begin{eqnarray*}
&&H^0(X, (m+1)(K_X+S+A))\\
&=&H^0(X, \omega_X(S+A+mL))\\
&\to& H^0(S, \omega_S(A|_S+mL|_S))\\
&=&H^0(S, (m+1)(K_S+A|_S))
\end{eqnarray*}
is surjective.
\end{proof}

\section{Calculation for the case $\kappa=0$}

In this section, we consider whether Theorem~\ref{trace-surface} 
holds for the case $\kappa(X, K_X+A)=0$. 
Let $X$ be a smooth projective surface and let $A$ be an ample \Z-divisor on $X$. 
Assume that $K_X+A$ is nef and $\kappa(X, K_X+A)=0$. 
By the abundance theorem (\cite[The second theorem in Introduction]{Fujita3}), 
we see $K_X+A\sim_{\mathbb Q} 0$. 
Then, $-K_X$ is ample. 
In particular, $X$ is a rational surface. 
In this case ${\rm Pic}(X)$ has no torsion, hence $K_X+A\sim 0$. 
We consider the following question. 

\begin{ques}\label{ques-del-Pezzo}
Let $X$ be a smooth projective surface such that 
$-K_X$ is ample. 
Is the trace map 
\begin{eqnarray*}
\Tr_X^e(-K_X):H^0(X, \omega_X(-p^eK_X))\to H^0(X, \omega_X(-K_X))
\end{eqnarray*}
surjective?
\end{ques}

If $K_X^2\geq 4$, then we obtain an affirmative answer.

\begin{prop}
Let $X$ be a smooth projective surface such that $-K_X$ is ample. 
If $K_X^2\geq 4$, then the trace map 
$$\Tr_X^e(-K_X):H^0(X, \omega_X(-p^eK_X))\to H^0(X, \omega_X(-K_X))$$
is surjective for every $e\in \mathbb Z_{>0}$. 
\end{prop}

\begin{proof}
Since $h^0(X, \omega_X(-K_X))=1$, 
it is sufficient to show that   
the trace map 
$$\Tr_X^e(-K_X):H^0(X, \omega_X(-p^eK_X))\to H^0(X, \omega_X(-K_X))$$
is a nonzero map. 
Since $K_X^2\geq 4$, $X$ is obtained by blowing up $\mathbb P^2$ at  $\leq 5$ points. 
Therefore, we can find a smooth conic $C_0$ passing through these points. 
Let $L_0$ be a line which does not pass through these points. 
Let $C$ and $L$ be the proper transforms of $C_0$ and $L_0$, respectively. 
We see that $L|_C$ is ample, $H^0(C, \omega_C(L|_C))\neq 0$ and $L$ is nef and big. 
Since $C+L\in |-K_X|$, 
we can apply Corollary~\ref{nonzero-criterion}.
\end{proof}

If $X$ is $F$-split, 
then the trace map $\Tr_X^e(-K_X)$ in Question~\ref{ques-del-Pezzo} is surjective. 
Note that, by \cite[Example~5.5]{Hara} and \cite[Proposition~4.10]{Smith}, if $K_X^2\geq 4$, then $X$ is $F$-split. 
However, since \cite{Hara} contains no explicit proof, we decided to include the above proof. 
Moreover, \cite[Example~5.5]{Hara} and \cite[Proposition~4.10]{Smith} show that, if $K_X^2=3$ and $X$ is not $F$-split, then 
$X$ is a Fermat type cubic surface in characteristic two. 
Indeed, this example gives a negative answer to Question~\ref{ques-del-Pezzo} as follows. 

\begin{thm}\label{trace-cex}
Let ${\rm char}\,k=p=2$. 
Consider $\mathbb P^3$ and let $[x:y:z:w]$ be the homogeneous coordinates. 
Let 
$$X:=\{[x:y:z:w]\in \mathbb P^3\,|\,x^3+y^3+z^3+w^3=0\}.$$
Then the trace map 
$$\Tr_X^e(-K_X):H^0(X, \omega_X(-2^eK_X))\to H^0(X, \omega_X(-K_X))$$
is the zero map for every $e\in\mathbb Z_{>0}$.
\end{thm}

\begin{proof}
By Lemma~\ref{e-vs-e+1}, we may assume $e=1$. 
By Lemma~\ref{trace-diagram}, we obtain the following commutative diagram 
$$\begin{CD}
0@>>> F_*\omega_{\mathbb P^3}@>>> F_*(\omega_{\mathbb P^3}(X))@>>> F_*\omega_X@>>> 0\\
@. @VVV @VVV @VVV\\
0@>>> \omega_{\mathbb P^3}@>>> \omega_{\mathbb P^3}(X)@>>> \omega_X@>>> 0.
\end{CD}$$
Tensoring $\MO_{\mathbb P^3}(-K_{\mathbb P^3}-X)$ and taking $H^0$, we obtain  
$$\begin{CD}
H^0(\mathbb P^3, \omega_{\mathbb P^3}(X-2K_{\mathbb P^3}-2X))@>\beta>> H^0(X, \omega_{X}(-2K_{X}))\\
@VV\Tr_{\mathbb P^3, X}(-K_{\mathbb P^3}-X)V @VV\Tr_X(-K_X)V\\
H^0(\mathbb P^3, \omega_{\mathbb P^3}(X-K_{\mathbb P^3}-X))@>\alpha>> H^0(X, \omega_{X}(-K_{X})).
\end{CD}$$
Since $H^1(\mathbb P^3, L)=0$ for an arbitrary invertible sheaf $L$, $\beta$ is surjective. 
Therefore, it is sufficient to prove that the trace map 
$\Tr_{\mathbb P^3, X}(-K_{\mathbb P^3}-X)$
is the zero map. 
By Lemma~\ref{D-linearequ}, we obtain 
$$\Tr_{\mathbb P^3, X}(-K_{\mathbb P^3}-X)=\Tr_{\mathbb P^3, X}(H)$$
where $H$ is defined by 
$$H:=\{[x:y:z:w]\in \mathbb P^3\,|\,w=0\}.$$
Thus, we show that 
$$\Tr:=\Tr_{\mathbb P^3, X}(H):H^0(\mathbb P^3, \omega_{\mathbb P^3}(X+2H))\to 
H^0(\mathbb P^3, \omega_{\mathbb P^3}(X+H))$$
is the zero map. 
Let us take a $k$-linear basis of $H^0(\mathbb P^3, \omega_{\mathbb P^3}(X+2H))$. 
Note that 
$$h^0(\mathbb P^3, \omega_{\mathbb P^3}(X+2H))=4.$$ 
Let $\Spec\,k[X, Y, Z]\subset \mathbb P^3$ be the affine open subset defined by $w\neq 0$. 
Consider the following four $3$-forms 
\begin{eqnarray*}
\eta_1&:=&\frac{1}{X^3+Y^3+Z^3+1}dX\wedge dY\wedge dZ\\
\eta_X&:=&\frac{X}{X^3+Y^3+Z^3+1}dX\wedge dY\wedge dZ\\
\eta_Y&:=&\frac{Y}{X^3+Y^3+Z^3+1}dX\wedge dY\wedge dZ\\
\eta_Z&:=&\frac{Z}{X^3+Y^3+Z^3+1}dX\wedge dY\wedge dZ.\\
\end{eqnarray*}
These are elements of $\omega_{k(X, Y, Z)}=(\omega_{\mathbb P^3})_{\xi}$ 
where $\xi$ is the generic point of $\mathbb P^3$. 
By a direct calculation, these four elements are linearly independent and 
$\eta_1, \eta_X,\eta_Y, \eta_Z\in H^0(\mathbb P^3, \omega_{\mathbb P^3}(X+2H)).$ 
In particular, these four elements form a $k$-linear basis of $H^0(\mathbb P^3, \omega_{\mathbb P^3}(X+2H))$. 
The trace map is a $p^{-1}$-linear map, that is, 
for $a, b, c, d\in k$,
\begin{eqnarray*}
&&\Tr(a\eta_1+b\eta_X+c\eta_Y+d\eta_Z)\\
&=&a^{\frac{1}{p}}\Tr(\eta_1)+b^{\frac{1}{p}}\Tr(\eta_X)+c^{\frac{1}{p}}\Tr(\eta_Y)+d^{\frac{1}{p}}\Tr(\eta_Z).
\end{eqnarray*}
Thus, it is sufficient to show 
$$\Tr(\eta_1)=\Tr(\eta_X)=\Tr(\eta_Y)=\Tr(\eta_Z)=0.$$
Let us only prove $\Tr(\eta_X)=0$. 
This follows from 
\begin{eqnarray*}
&&(\Tr(\eta_X))|_{\Spec\,k[X, Y, Z]}\\
&=&\Tr\left(\frac{X}{X^3+Y^3+Z^3+1}dX\wedge dY\wedge dZ\right)\\
&=&\Tr\left(\frac{X(X^3+Y^3+Z^3+1)}{(X^3+Y^3+Z^3+1)^2}dX\wedge dY\wedge dZ\right)\\
&=&\frac{1}{X^3+Y^3+Z^3+1}\Tr\left((X^4+XY^3+XZ^3+X)dX\wedge dY\wedge dZ\right)\\
&=&0.
\end{eqnarray*}
The last equality follows from Remark~\ref{trace-local}. 
\end{proof}

We do not know whether the conclusion of Theorem~\ref{threefold-extension} 
holds when $\kappa(S, K_S+A|_S)=0$. 
However, the following example shows that when this is the case, 
we can not argue as in the proof of the theorem. 

\begin{ex}
Let ${\rm char}\,k=p=2$. 
Then, there exist smooth projective threefold $X$ over $k$, 
a smooth prime divisor $S_0$ on $X$ and 
an ample \Z-divisor $A$ on $X$ which satisfy the following properties. 
\begin{enumerate}
\item{$|K_X+S_0+A|$ is basepoint free. }
\item{The natural restriction map 
$$H^0(X, m(K_X+S_0+A))\to H^0(S_0, m(K_{S_0}+A|_{S_0}))$$
is surjective for every $m\in\mathbb Z_{>0}$.}
\item{
The trace map $\Tr_{S_0}^e(A|_{S_0}+m(K_{S_0}+A|_{S_0}))$
{\Small $$H^0(S_0, \omega_{S_0}(2^e(A|_{S_0}+m(K_{S_0}+A|_{S_0}))))\to 
H^0(S_0, \omega_{S_0}(A|_{S_0}+m(K_{S_0}+A|_{S_0})))$$}
is the zero map for every $m\in \mathbb Z_{>0}$ and for every $e\in\mathbb Z_{>0}.$}
\end{enumerate}
\end{ex}

\begin{proof}
Let $S$ be the surface in Theorem~\ref{trace-cex} and 
let $A_S:=-K_S$. 
Take an arbitrary smooth projective curve $C$ and 
fix an arbitrary ample \Z-divisor $A_C$ on $C$. 
Let $X:=S\times C$ and let $\pi_S$ and $\pi_C$ be 
the projections onto the first and second component, respectively. 
Fix a point $c_0\in C$ and let $S_0:=S\times \{c_0\}$. 
Let 
$$A:=\pi_S^*A_S+\pi_C^*A_C.$$
Note that $A|_{S_0}=-K_{S_0}$. 
Thus, (3) follows from Theorem~\ref{trace-cex}. 
The assertion in (1) follows from 
\begin{eqnarray*}
K_X+S_0+A&=&\pi_S^*(K_S+A_S)+\pi_C^*(K_C+c_0+A_C)\\
&=&\pi_C^*(K_C+c_0+A_C).
\end{eqnarray*}
Therefore, in order to conclude the proof, 
it is sufficient to show (2). 
This follows from 
{\small 
\begin{eqnarray*}
&&H^1(X, K_X+A+(m-1)(K_X+S_0+A))\\
&=&H^1(X, \pi_C^*(K_C+A_C+(m-1)(K_C+c_0+A_C)))\\
&\simeq&H^1(S, \mathcal O_S)\otimes_k H^0(C, K_C+A_C+(m-1)(K_C+c_0+A_C))\\
&\oplus &H^0(S, \mathcal O_S)\otimes_k H^1(C, K_C+A_C+(m-1)(K_C+c_0+A_C))\\
&=&0.
\end{eqnarray*}}
The last equality holds since $H^1(S, \mathcal O_S)=0$ 
and 
$$\deg_C(A_C+(m-1)(K_C+c_0+A_C))>0.$$
\end{proof}


\section{Appendix: Extension theorem for surfaces}

For the surface case, we can freely use the minimal model theory 
(cf. \cite{Fujita3}, \cite{KK}, \cite{T2}). 
By using results obtained in \cite{Fujita3}, \cite{T2} and \cite{T3}, 
we can establish an analogue of \cite[Theorem~5.4.21]{HM} as follows.

\begin{thm}[Extension theorem]\label{surface-ext}
Let $X$ be a smooth projective surface and 
let $C$ be a smooth prime divisor on $X$. 
Let $\Delta:=C+B$, where 
$B$ is an effective \Q-divisor which satisfies the following properties: 
\begin{enumerate}
\item{$C\not\subset\Supp B$, $\llcorner B\lrcorner=0$, and 
$(X, \Delta)$ is plt, }
\item{$B$ is a big \Q-divisor, and}
\item{No prime component of $\Delta$ is contained in the stable base locus of $K_X+\Delta.$}
\end{enumerate}
Then, there exists an integer $m_0>0$ such that, for every integer $m>0$, 
the restriction map 
$$H^0(X, mm_0(K_X+\Delta))\to H^0(C, mm_0(K_X+\Delta)|_C)$$
is surjective.
\end{thm}

\begin{proof}
\setcounter{step}{0}
\begin{step}
In this step we prove that, if $E$ is a curve in $X$ such that 
$E^2<0$ and $(K_X+\Delta)\cdot E<0$, then 
the following three assertions hold:  
\begin{enumerate}
\item[(a)]{$K_X\cdot E=E^2=-1$,}
\item[(b)]{$E$ is not a prime component of $\Delta$, and}
\item[(c)]{$E\cdot C=0$.} 
\end{enumerate}

Since $E^2<0$, we obtain $(K_X+E)\cdot E\leq (K_X+\Delta)\cdot E<0$. 
Then, there exists 
a birational morphism $f:X\to Y$ to a normal \Q-factorial surface $Y$ such that $\Ex(f)=E$ 
(cf. \cite[Theorem~6.2]{T2}). 
Let $\Delta_Y:=f_*\Delta$ and we define $d\in\mathbb Q$ by 
$$K_X+\Delta=f^*(K_Y+\Delta_Y)+dE.$$
The inequalities $(K_X+\Delta)\cdot E<0$ and $E^2<0$ imply $d>0.$ 
We can find an integer $l>0$ such that $l(K_Y+\Delta_Y)$ is Cartier. 
Then, $E$ is a fixed component of 
$$l(K_X+\Delta)=f^*(l(K_Y+\Delta_Y))+dlE.$$
We deduce that the assumption (3) implies (b). 
This gives $E\cdot \Delta\geq 0.$ 
Thus, the assertion (a) follows from 
$$K_X\cdot E\leq (K_X+\Delta)\cdot E<0.$$
Let us show (c). 
If $E\cdot C>0$, then $E\cdot C\geq 1.$ 
This implies the following contradiction 
$$0>(K_X+\Delta)\cdot E=K_X\cdot E+C\cdot E+B\cdot E\geq -1+1+0=0.$$
\end{step}

\begin{step}
In this step we prove that 
we may assume that $K_X+\Delta$ is nef. 

Assume that $K_X+\Delta$ is not nef. 
Then, there exists a curve $E$ such that $(K_X+\Delta)\cdot E<0$. 
By (3), there exists an integer $l>0$ such that $|l(K_X+\Delta)|\neq\emptyset.$ 
This implies $E^2<0$. 
We see that $E$ is a $(-1)$-curve by Step~1. 
Let $f:X\to Y$ be the contraction of $E$. 
Let 
$$\Delta_Y:=f_*\Delta\,,\,\,C_Y:=f_*C\,,\,\,B_Y:=f_*B.$$
We can check that $Y$ and these divisors also satisfy the conditions (1), (2) and (3). 
Let $m_1>0$ be an integer such that $m_1\Delta$ is a \Z-divisor. 
Then we have 
$$m_1(K_X+\Delta)=f^*(m_1(K_Y+\Delta_Y))+eE$$
for some $e\in\mathbb Z_{>0}.$ 
Let $n$ be an arbitrary positive integer. 
Since $f_*\mathcal O_X=\mathcal O_Y$, we have 
{\Small $$f_*(\mathcal O_X(nm_1(K_X+\Delta)))\simeq f_*(\mathcal O_X(nm_1f^*(K_Y+\Delta_Y)))\simeq 
\mathcal O_Y(nm_1(K_Y+\Delta_Y)).$$}
Since $C\cap E=\emptyset$ by Step~1, we have 
{\Small $$f_*(\mathcal O_C(nm_1(K_X+\Delta)))\simeq f_*(\mathcal O_C(nm_1f^*(K_Y+\Delta_Y)))\simeq 
\mathcal O_{C_Y}(nm_1(K_Y+\Delta_Y)).$$}
We conclude that we have the following commutative diagram 
$$\begin{CD}
H^0(X, nm_1(K_X+\Delta))@>>> H^0(C, nm_1(K_X+\Delta)|_C)\\
@| @|\\
H^0(Y, nm_1(K_Y+\Delta_Y))@>>> H^0(C_Y, nm_1(K_Y+\Delta_Y)|_{C_Y})
\end{CD}$$
where the horizontal arrows are the natural restriction maps. 
Thus, we can reduce the problem on $X$ to the problem on $Y$. 
After repeating this argument finitely many times, we reduce to the
case when $K_X+\Delta$ is nef. 
\end{step}

\begin{step}
By the abundance theorem (\cite[the second theorem in Introduction]{Fujita3}), 
$K_X+\Delta$ is semi-ample. 
Let 
$$f:=\varphi_{|m_2(K_X+\Delta)|}:X\to R$$
for some $m_2\in\mathbb Z_{>0}$ 
such that $f_*\mathcal O_X=\mathcal O_R.$ 
In this step, 
we prove $f_*\mathcal O_C=\mathcal O_{f(C)}$. 

Assume $f_*\mathcal O_C\neq \mathcal O_{f(C)}$. 
We run the $(K_X+\{\Delta\})$-MMP, where 
$\{\Delta\}$ denotes the fractional part of $\Delta$. 
By \cite[Proposition~2.8]{T3}, 
there exist morphisms 
$$X\overset{g}\to V\to R$$
where $V$ is a smooth projective curve such that 
a general fiber $G$ of $g$ satisfies 
$G\simeq\mathbb P^1$ and $\llcorner\Delta\lrcorner\cdot G=2.$ 
Note that $B=\{\Delta\}$ and $C=\llcorner\Delta\lrcorner$. 
Since $G$ is a fiber and $B$ is big, we deduce $G\cdot B>0$. 
Thus we obtain the following contradiction:  
$$0=(K_X+\Delta)\cdot G=(K_X+B)\cdot G+2>(K_X+G)\cdot G+2=0.$$
\end{step}

\begin{step}
In this step we prove the assertion in the theorem. 
Let 
$$f:=\varphi_{|m_2(K_X+\Delta)|}:X\to R$$
such that $f_*\mathcal O_X=\mathcal O_R$ and 
let $f(C)=:D.$ 
By Step~3, we have $f_*\mathcal O_C=\mathcal O_D.$ 
Let $H$ be an ample Cartier divisor on $R$ such that 
$m_2(K_X+\Delta)=f^*H.$ 
By the Serre vanishing theorem, 
we can find $m_3\in\mathbb Z_{>0}$ such that 
$$H^0(R, mm_3H)\to H^0(D, mm_3 H)$$
is surjective for every $m\in \mathbb Z_{>0}.$
Since $f_*\mathcal O_X=\mathcal O_R$ and 
$f_*\mathcal O_C=\mathcal O_D$, 
we have the following commutative diagram 
$$\begin{CD}
H^0(X, mm_2m_3(K_X+\Delta))@>>> H^0(C, mm_2m_3(K_X+\Delta)|_C)\\
@| @|\\
H^0(R, mm_3H)@>>> H^0(D, mm_3H|_{D}).
\end{CD}$$
This implies the assertion in the theorem. 
\end{step}
\end{proof}

\end{document}